\documentclass[a4paper, 11pt]{amsart}
\usepackage{tikz-cd}
\usepackage{amsmath,amssymb,amsthm}
\usepackage[a4paper,hmargin=1.5cm,vmargin=2cm]{geometry}
\usepackage[english]{babel}
\usepackage{amsfonts}
\usepackage{amscd}
\usepackage{hyperref}
\usepackage{extpfeil}
\usepackage{tikz}
\usepackage[all,cmtip]{xy}
\usepackage{mathrsfs}
\usetikzlibrary{arrows}
\usepackage{enumitem}
\usepackage{color}

\newcommand{\bD}{\mathbf{D}}
\newcommand{\bE}{\mathbf{E}}

\DeclareMathOperator{\Ker}{ker}
\DeclareMathOperator{\Rk}{rk}
\DeclareMathOperator{\Image}{im}
\DeclareMathOperator{\Deg}{deg}
\DeclareMathOperator{\Coker}{coker}
\DeclareMathOperator{\Supp}{supp}
\DeclareMathOperator{\Det}{det}

\DeclareMathOperator{\Pic}{Pic}
\DeclareMathOperator{\Ev}{ev}

\DeclareMathOperator{\Hom}{Hom}

\DeclareMathOperator{\Ext}{Ext}
\DeclareMathOperator{\Bs}{Bs}

\numberwithin{equation}{section}

\theoremstyle{plain}

\newtheorem{theorem}{Theorem}[section]

\newtheorem{lemma}[theorem]{Lemma}
\newtheorem{proposition}[theorem]{Proposition}
\newtheorem{corollary}[theorem]{Corollary}

\theoremstyle{definition}

\newtheorem{definition}[theorem]{Definition}

\newtheorem{setup}[theorem]{Setup}

\theoremstyle{remark}

\newtheorem{remark}[theorem]{Remark}

\title{Brill--Noether Generality of Curves and K3 Surfaces}

\begin{document}
	
	\author{Irina Shatova}
	\address{Laboratory of Algebraic Geometry,
			National Research University Higher School of Economics,
			Russian Federation}
	\email{ioshatova@edu.hse.ru}
	\thanks{The study has been funded within the framework of the HSE University Basic Research Program (HSE-BR-2025-061).}

	\date{}

	\begin{abstract}
	Lazarsfeld proved Brill--Noether generality of any smooth curve in the linear system~$|H|$
	where~$(X,H)$ is a polarized K3 surface with~$\Pic(X) = \mathbb{Z}\cdot H$.
	Mukai introduced the notion of Brill--Noether generality for quasi-polarized K3 surfaces.
	We prove Brill--Noether generality of any smooth curve in the linear system~$|H|$
	where~$(X,H)$ is a Brill--Noether general quasi-polarized K3 surface. We also prove that Petri homomorphism is an isomorphism for each extremal line bundle on a general curve $C \in |H|$.
	\end{abstract}

	\maketitle

	\tableofcontents
	
	\section{Introduction}

	Brill--Noether theory is one of the central subjects of the theory of linear systems on smooth proper curves.
	Recall that, given a line bundle~$A$ on a smooth proper curve~$C$ of genus~$g$, the \emph{Brill--Noether number of $A$}
	is defined by
	\begin{equation*}
	\rho(A) \coloneqq g - h^0(A)h^1(A).
	\end{equation*}
	It is equal to the expected dimension of the \emph{Brill--Noether locus} in~$\Pic(C)$ parameterizing line bundles~$A'$
	with~$\deg(A') = \deg(A)$ and~$h^i(A') \ge h^i(A)$ for~$i \in \{0,1\}$.
	In particular, if~$\rho(A) < 0$, the expected dimension of the Brill--Noether locus is negative,
	and it is natural to expect that, if~$C$ is general, the locus is empty.
	This is indeed the case, and the usual way to formulate this is by using the following

	\begin{definition}\label{definition of a BN curve}
		A smooth proper curve $C$ is \emph{Brill--Noether general} if
		\begin{equation*}
		\rho(A) \geq 0
		\end{equation*}
		for any line bundle $A$ on $C$.
	\end{definition}

	\begin{theorem}[{\cite{MR563378, lazarsfeld1986brill}}]
	\label{brill--noether main theorem}
		For any~$g \ge 0$ a general curve of genus~$g$ is Brill--Noether general.
	\end{theorem}

	All proofs of this theorem are based on the openness of the locus of Brill--Noether general curves in families;
	so, to prove the theorem it is enough to find one Brill--Noether general curve of genus~$g$ for each~$g \ge 0$.

	The first proof of Theorem~\ref{brill--noether main theorem}, due to Griffiths and Harris~\cite{MR563378}, was quite complicated;
	it used degenerations of smooth curves to nodal curves.
	
	Later, Lazarsfeld found another proof \cite{lazarsfeld1986brill}. He showed that all smooth curves in the linear system~$|H|$
	on a polarized K3 surface~$(X,H)$ with
	\begin{equation}
	\label{eq:pic1}
	\Pic(X) = \mathbb{Z} \cdot H
	\end{equation}
	are Brill--Noether general.
	The key notion in the proof of Lazarsfeld is the following.

	\begin{definition}\label{definition of L bundle}
		Let $X$ be a smooth K3 surface, let $i \colon C \hookrightarrow X$ be an embedding of a smooth curve,
		and let~$A$ be a globally generated line bundle on $C$.
		The \emph{Lazarsfeld bundle associated with the pair $(A,C)$}, denoted~$F_{A,C}$,
		is the kernel of the evaluation map $\Ev_{A,C} \colon H^0(A) \otimes \mathcal{O}_X \twoheadrightarrow i_* A$.
		In other words, the Lazarsefeld bundle~$F_{A,C}$ is defined by the exact sequence
		\begin{equation*}
		0 \longrightarrow F_{A,C} \longrightarrow H^0(A) \otimes \mathcal{O}_X \xrightarrow{\Ev_{A,C}} i_*A \longrightarrow 0.
		\end{equation*}
	\end{definition}

	To prove Theorem~\ref{brill--noether main theorem}, Lazarsfeld observed that
	\begin{equation*}
	\chi(F_{A,C}^\vee \otimes F_{A,C}) = 2 - 2\rho(A).
	\end{equation*}
	If~$\rho(A) < 0$ it follows that~$F_{A,C}$ has an endomorphism~$\phi$ with~$0 < \Rk{\phi} < \Rk{F_{A,C}}$, and therefore
	\begin{equation*}
	H = -c_1(F_{A,C}) = -c_1(\Ker{\phi}) - c_1(\Image{\phi})
	\end{equation*}
	is a decomposition of the polarization~$H$ into the sum of two effective non-zero divisors,
	contradicting assumption~\eqref{eq:pic1} (see~\cite[Lemma~1.3]{lazarsfeld1986brill} for details).
	
	Recall that for a curve $C$ and a line bundle $A$ on it the \emph{Petri homomorphism} is the homomorphism $$\mu_0: H^0(A) \otimes H^0(\omega_C \otimes A^{\vee}) \longrightarrow H^0(\omega_C)$$ given by multiplication of global sections. The curve $C$ is called \emph{Brill--Noether--Petri general} if Petri homomorphism is injective for all line bundles $A$ on $C$. Lazarsfeld also proves that on a K3 surface $X$ satisfying~\eqref{eq:pic1} a general curve $C \in |H|$ is Brill--Noether--Petri general (see~\cite[Theorem]{lazarsfeld1986brill}).
	
	The arguments of Lazarsfeld are very simple and beautiful;
	however, assumption~\eqref{eq:pic1} on which they rely heavily is a little bit too restrictive for some applications.
	For instance, in a family of polarized K3 surfaces, the set of those that satisfy~\eqref{eq:pic1}
	is not Zariski open; in fact, it is the complement of a countable union of divisors.
	A better replacement for assumption~\eqref{eq:pic1} was suggested by Mukai.

	Recall that a \emph{quasi-polarization} on a smooth K3 surface~$X$ is a big and nef divisor class~$H$ and its genus is defined by
	\begin{equation*}
	g(X, H) = \frac{1}{2}H^2 + 1.
	\end{equation*}
	By the adjunction formula~$g(X, H)$ is equal to the genus of any smooth curve in the linear system~$|H|$.

	\begin{definition}[{\cite[Definition~3.8]{Mukai}}]
	\label{definition of a BN general K3}
		A quasi-polarized K3 surface $(X,H)$ is called \emph{Brill--Noether general}
		if for any decomposition $H = D_1 + D_2$ with $D_i > 0$ (i.e., $D_i$ are effective and nonzero)
		one has
		\begin{equation*}
		g(X, H) - h^0(D_1)h^0(D_2) \geq 0.
		\end{equation*}
	\end{definition}

	In contrast to~\eqref{eq:pic1}, it is not hard to see that this property is Zariski open, see~\cite[Lemma~2.8]{greer2015picard}.
	Still, it was expected that the analog of the Lazarsfeld's theorem holds for Brill--Noether general K3 surfaces
	(see, e.g., \cite[Remark~10.2]{JK} for the statement
	and~\cite[Theorem~1]{Hab} for the proof under the assumption~$g(X,H) \le 19$).
	Our main result is the general case of this expectation.

	\begin{theorem}\label{main theorem}
		Let $(X,H)$ be a Brill--Noether general quasi-polarized K3 surface.
		Then every smooth curve $C \in |H|$ is Brill--Noether general.
	\end{theorem}

	We expect that Theorem~\ref{main theorem} may be quite useful for the study of projective models of K3 surfaces and Fano varieties.
	For instance, it may help extending the description of~\cite[Theorem~1.3]{bayer2025mukai}
	from K3 surfaces of Picard rank~1 to Brill--Noether general K3 surfaces,
	and as a consequence, extending the description of~\cite[Theorem~1.1]{bayer2025mukai}
	from Fano varieties of Picard rank~1 to Brill--Noether general Fano varieties in the sense of~\cite[Definition~6.4]{Mukai}.

	To prove Theorem~\ref{main theorem} we also study the Lazarsfeld bundle~$F_{A,C}$.
	Namely, assuming that~$\rho(A) < 0$, so that~$F_{A,C}$ is not simple,
	we show that~$F_{A,C}$ has a ``harmonic filtration'' (see Definition~\ref{def:good-filtration}) of length~$n \ge 2$,
	whose properties are summarized in Lemma~\ref{properties of good filtration of L bundles}.
	It gives a decomposition
	\begin{equation*}
	H = D_1 + \dots + D_n
	\end{equation*}
	of the quasi-polarization~$H$ into a sum of divisor classes satisfying~$D_i^2 \ge 0$ and~$h^0(D_i) \ge 2$
	(we also use Lemma~\ref{sorting fixed components} here).
	Then we show (Proposition~\ref{conditions on a polarizing divisor}) that, for~$n \ge 3$,
	the Brill--Noether generality of~$(X,H)$ imposes serious constraints on the squares~$D_i^2$;
	explicitly, we prove that on Brill--Noether general K3 surface one must have~$n \le 4$, and, moreover,
	\begin{itemize}
	\item
	if~$n = 4$ then~$D_1^2 = D_2^2 = D_3^2 = D_4^2 = 0$, and
	\item
	if~$n = 3$ then, assuming that~$D_1^2 \ge D_2^2 \ge D_3^2$, one must have
	\begin{itemize}
	\item
	$D_3^2 = D_2^2 = 0$, or
	\item
	$D_3^2 = 0$, $D_2^2 = 2$, and~$D_1^2 \in \{2,4,6\}$.
	\end{itemize}
	\end{itemize}

	\begin{remark}
	We observe that in the critical situations of Proposition~\ref{conditions on a polarizing divisor},
	if we set~\mbox{$a_i \coloneqq \frac12D_i^2 + 1$}, the triples~$(a_1,a_2,a_3)$ for~$n = 3$ in the above list
	correspond to the lengths of the legs of the Dynkin diagrams of types~$\bD_{m+3}$ and~$\bE_{m}$ with~$m = \sum a_i + 1$.
	\end{remark}

	To rule out the remaining few cases we use the more subtle properties of harmonic filtrations.
	Namely, if~$F_{A,C} = F_1 \supsetneq F_2 \supsetneq \dots \supsetneq F_{n} \supsetneq F_{n+1} = 0$ is a harmonic filtration and
	\begin{equation*}
	v(F_i/F_{i+1}) = (r_i, -D_i, s_i),
	\qquad
	1 \le i \le n,
	\end{equation*}
	are the Mukai vectors of its associated graded factors, Lemma~\ref{properties of good filtration of L bundles} proves the equalities
	\begin{equation*}
	r_1 + \dots + r_n = h^0(A),
	\qquad
	s_1 + \dots + s_n = h^1(A),
	\end{equation*}
	and the inequalities
	\begin{equation*}
	r_is_i \le \frac12D_i^2 + 1,
	\qquad
	r_i \le r_{i+1} + \dots + r_n,
	\qquad\text{and}\qquad
	s_n \ge 1.
	\end{equation*}
	By Proposition~\ref{conditions on a polarizing divisor} we know that~$n \le 4$;
	we consider the remaining cases separately.
	The case where~$n = 2$ is quite simple, see Proposition~\ref{simple kernel and image},
	the case where~$n = 4$ is more complicated, see Proposition~\ref{theorem case n=4},
	and the hardest case is~$n = 3$, see Proposition~\ref{theorem case n=3}.
	Summarizing all this, we deduce Theorem~\ref{main theorem}.	
	
	We also study the Petri homomorphisms. In particular, we prove that on a general curve $C \in |H|$ on a Brill--Noether general K3 surface $(X,H)$ the Petri homomorphism is an isomorphism for any extremal (i.e. satisfying $\rho(A) = 0$) line bundle $A$ on $C$. 
	
	To state the precise result recall from~\cite{bayer2024} the following definition.
	
	\begin{definition}\label{defin-Mukai-class}
		\emph{A special Mukai class of type $(r,s)$} is a base-point-free divisor $\Xi$ on $X$ such that 
		\begin{center}
			$\Xi \cdot H = (r-1)(s+1), \qquad h^0(\Xi) = r \qquad and \qquad h^1(\Xi) = h^2(\Xi) = 0.$ 
		\end{center}
	\end{definition}
	
	The theorem that we are going to prove is the following.
	
	\begin{theorem}\label{Petri main}
		Let $(X,H)$ be a Brill--Noether general quasi-polarized K3 surface
		of genus $g = g(H) = r \cdot s$ with $r,s \ge 2$. Suppose that $(X,H)$ does not contain a special Mukai class of type $(r,s)$.
		If~$C \in |H|$ is a general curve and~$A$ is an extremal line bundle on~$C$ with $h^0(A) = r$ and $h^1(A) = s$
		then the Petri homomorphism
		$$\mu_0 : H^0(A) \otimes H^0(A^{\vee} \otimes \omega_C) \longrightarrow H^0(\omega_C)$$
		is an isomorphism.
	\end{theorem}

	Note that any smooth curve~$C$ in~$|H|$ is Brill--Noether general by Theorem~\ref{main theorem},
	hence line bundles~$A$ with $h^0(A) = r$ and $h^1(A) = s$ on~$C$ are parameterized by a zero-dimensional Brill--Noether locus;
	it follows from Theorem~\ref{Petri main} that this locus is reduced.
	
	The paper is organized as follows.
	We start with a short preliminary Section~\ref{sec Preliminaries}.
	In Section~\ref{sec Divisors on Brill--Noether General K3 Surface}
	we study some implications of the Brill--Noether property of~$(X,H)$
	for decompositions of~$H$ into a sum of three or more effective divisors.
	More precisely, in Subsection~\ref{subsec Linear Systems on K3 surfaces}
	we recall some basic results on linear systems on K3 surfaces from~\cite{72768972-f7b8-3f72-b8d3-5e64ee0f3017},
	and in Subsection~\ref{subsec The necessary conditions for Brill--Noether generality}
	we prove the main result of this section, Proposition~\ref{conditions on a polarizing divisor}. In Section~\ref{sec Lazarsfeld Bundles} we show that Lazarsfeld bundles always admit a harmonic filtration and describe its properties.
	In Section~\ref{sec proof} we combine our results
	from Sections~\ref{sec Divisors on Brill--Noether General K3 Surface} and~\ref{sec Lazarsfeld Bundles}
	to prove Theorems~\ref{main theorem}. 
	Finally, in Section~\ref{sec Petri} we prove Theorem~\ref{Petri main}.
	
	\subsection*{Acknowledgemets}

	I am deeply grateful to my advisor, Alexander Kuznetsov, for his immense patience, unwavering support and careful attention to this work.

	\section{Preliminaries}\label{sec Preliminaries}
	
	In this section we recall some basic facts about Mukai vectors of sheaves on K3 surfaces
	and the Brill--Noether property.
	
	If $F$ is a coherent sheaf on a K3 surface $X$, the \emph{Mukai vector} of $F$ is the triple
	\begin{equation*}
	v(F) = \langle r(F), c_1(F), s(F) \rangle \in \mathbb{Z} \oplus \Pic{X} \oplus \mathbb{Z},
	\end{equation*}
	where $r(F) = \Rk{F}$ is the rank of $F$, $c_1(F)$ is the first Chern class of $F$, and
	\begin{equation*}
	s(F) = \Rk{F} + ch_2(F).
	\end{equation*}
	\emph{The Mukai pairing} on $\mathbb{Z} \oplus \Pic{X} \oplus \mathbb{Z}$ is defined as follows:
	\begin{equation*}
	(v(F), v(G)) := c_1(F) \cdot c_1(G) -r(F)s(G) - r(G)s(F).
	\end{equation*}
	By the Hirzebruch--Riemann--Roch theorem, one has~$s(F) = \chi(F) - \Rk{F}$ and
	\begin{equation*}
	(v(F), v(G)) = - \chi(F,G) = - \dim{\Hom}(F,G) + \dim{\Ext}^1(F,G) - \dim{\Ext}^2(F,G).
	\end{equation*}
	Recall that on a K3 surface~$\Ext^2(F,G) \cong \Hom(G,F)^\vee$ by Serre duality; hence,
	\begin{equation*}
	(v(F), v(F)) = \dim{\Ext}^1(F,F) - 2 \dim{\Hom}(F,F).
	\end{equation*}

	The next observation will be frequently used.
	
	\begin{lemma}\label{formula simple sheaf Mukai vector}
		If a sheaf $F$ is simple then
		\begin{equation}
			r(F)s(F) \leq \frac{1}{2}c_1(F)^2 + 1.
		\end{equation}
	\end{lemma}
	
	\begin{proof}
		Indeed, $(v(F), v(F)) =  c_1(F)^2 - 2r(F)s(F) = - 2 + \dim{\Ext}^1(F,F) \geq -2.$
	\end{proof}

	We will use the following well-known result (for a proof see, e.g., \cite[Lemma~3.7(a)]{bayer2024}).

	\begin{lemma}\label{pre lem:cc}
		Suppose that $G \hookrightarrow \mathcal{O}_X^{\oplus n}$ is an embedding.
		Then $c_1(G) \leq 0$; moreover, if $c_1(G) = 0$ and $G$ is saturated then $G \cong \mathcal{O}_X^{\oplus s}$.
	\end{lemma}

	\begin{lemma}\label{lem:cc}
	Suppose that $G \hookrightarrow \mathcal{O}_X^{\oplus n}$ is an embedding.
	If $c_1(G) = 0$ and $G$ is locally free then $G \cong \mathcal{O}_X^{\oplus s}$.
	\end{lemma}

	\begin{proof}
		Consider the saturation $\widehat{G} \subset \mathcal{O}_X^{\oplus n}$ of $G$.
		Since $T = \widehat{G} / G$ is a torsion sheaf, $c_1(T) \ge 0$;
		therefore, we have~$c_1(\widehat{G}) = c_1(G) + c_1(T) \ge c_1(G) = 0$.
		But $\widehat{G} \subset \mathcal{O}_X^{\oplus n}$, hence~$\widehat{G} \cong \mathcal{O}_X^{\oplus s}$ by Lemma~\ref{pre lem:cc}.
		Thus, $G \hookrightarrow \mathcal{O}_X^{\oplus s}$ is an embedding of locally free sheaves
		of the same rank and first Chern class, hence it is an isomorphism.

	\end{proof}

	Recall Definitions~\ref{definition of a BN curve} and~\ref{definition of a BN general K3}
	of a Brill--Noether general curve and of a Brill--Noether general K3 surface.
	It is well known (see, e.g., \cite[Remark~10.2]{JK}) that one property implies the other.
	
	\begin{proposition}\label{prop:converse}
		Let $(X,H)$ be a quasi-polarized K3 surface and let $C\in|H|$ be a smooth curve.
		If $C$ is Brill--Noether general then so is~$(X,H)$.
	\end{proposition}
	
	\begin{proof}
		We denote the genus $g(X,H) = g(C)$ simply by $g$.
		Suppose for a contradiction that the quasi-polarized K3 surface~$(X,H)$ is not Brill--Noether general,
		i.e., there is a decomposition $H = D_1 + D_2$ with $D_i > 0$
		and~$h^0(D_1)h^0(D_2) > g$.
		Note that $D_1 - H = -D_2$ and $D_2 - H = -D_1$. Therefore, the restriction exact sequences take the form:
		\begin{align}\label{restriction sequence 1}
			0 \longrightarrow \mathcal{O}(-D_2) \longrightarrow \mathcal{O}(D_1) \longrightarrow \mathcal{O}_C(D_1) \longrightarrow 0,
		\\
		\label{restriction sequence 2}
			0 \longrightarrow \mathcal{O}(-D_1) \longrightarrow \mathcal{O}(D_2) \longrightarrow \mathcal{O}_C(D_2) \longrightarrow 0.
		\end{align}
		We denote $A \coloneqq \mathcal{O}_C(D_1)$;
		then $\mathcal{O}_C(D_2) = \mathcal{O}_C(-D_1) \otimes \mathcal{O}_C(H) = A^{\vee} \otimes \omega_C$.
		Since $D_i > 0$ by assumption, we have $h^0(-D_1) = h^0(-D_2) = 0$
		and it follows from~\eqref{restriction sequence 1} and~\eqref{restriction sequence 2}
		that $h^0(A) \geq h^0(D_1)$ and \mbox{$h^0(A^{\vee} \otimes \omega_C) \geq h^0(D_2)$}.
		Thus, we obtain
		\begin{equation*}
		h^0(A)h^1(A) = h^0(A)h^0(A^{\vee} \otimes \omega_C) \geq h^0(D_1)h^0(D_2) > g.
		\end{equation*}
		This contradicts the assumption that the curve~$C$ is Brill--Noether general.
	\end{proof}
		
	The following lemma allows us to reduce Theorems~\ref{main theorem}
	to the case of globally generated line bundle with globally generated adjoint. We denote by $\Bs{(A)}$ the base locus of a line bundle $A$.
		
	\begin{lemma}\label{global generation of L bundle^*}
		Suppose that there is a line bundle~$A$ on a smooth curve $C$ such that~$\rho(A) < 0$.
		Then there is a globally generated line bundle $\bar{A}$ on $C$ with globally generated adjoint bundle~$\bar{A}^{\vee} \otimes \omega_C$
		such that~$\rho(\bar{A}) < 0$.
	\end{lemma}
	
	\begin{proof}
		First, we check that for any line bundle $B$ on a curve $C$ of genus $g$ the product $h^0(B)h^1(B)$ is uniformly bounded.
		Indeed, we can assume that $h^0(B) > 0$ and~$h^1(B) > 0$.
		Then there is an exact sequence
		\begin{equation*}
		0 \to \mathcal{O}_C \to B \to F \to 0,
		\end{equation*}
		where~$\dim(\Supp(F)) = 0$. Hence,~$h^1(F) = 0$; therefore,~$h^1(B) \le h^1(\mathcal{O}_C) = g$.
		Applying Serre duality, we also see that $h^0(B) \leq g$.
		This gives the required uniform bound~$h^0(B)h^1(B) \leq  g^2$.
		
		Now, let~$A$ be a line bundle on the curve~$C$ such that~$\rho(A) < 0$.
		If $\Bs{(A)} \neq \varnothing$ and $p \in \Bs{(A)}$, we take $A' \coloneqq A(-p)$;
		then $h^0(A') = h^0(A)$ and $h^1(A') > h^1(A)$.
		Similarly, if $\Bs{(A^{\vee} \otimes \omega_C)} \neq \varnothing$ and $p \in \Bs{(A^{\vee} \otimes \omega_C)}$,	we take $A' \coloneqq A(p)$;
		then $h^0(A') > h^0(A)$ and $h^1(A') = h^1(A)$.
		In either case we have
		\begin{equation*}
		h^0(A')h^1(A') > h^0(A)h^1(A),
		\end{equation*}
		hence~$\rho(A') < \rho(A) < 0$.
		Therefore, iterating this procedure and using the boundedness of the product of~$h^0$ and~$h^1$,
		we see that we will eventually arrive at a line bundle $\bar{A}$
		such that $\bar{A}$ and $\bar{A}^{\vee} \otimes \omega_C$ are globally generated
		and~$\rho(\bar{A}) < 0$, as required.
	\end{proof}

	\section{Divisors on Brill--Noether general K3 surface}
	\label{sec Divisors on Brill--Noether General K3 Surface}

	The Brill--Noether property of a quasi-polarized K3 surface~$(X,H)$ is defined in terms of decompositions of~$H$
	into a sum of two effective divisor classes.
	The main result of this section, Proposition~\ref{conditions on a polarizing divisor},
	provides necessary conditions for Brill--Noether generality in terms of decomposition of~$H$
	into a sum of three or more divisor classes.
	This serves as one of the two crucial ingredients in the proof of Theorem~\ref{main theorem}.
	
	\subsection{Linear systems on K3 surfaces}\label{subsec Linear Systems on K3 surfaces}

	In this subsection we briefly recall the fundamental results from~\cite{72768972-f7b8-3f72-b8d3-5e64ee0f3017}
	about linear systems on K3 surfaces and derive some basic consequences from them.

	The following property of linear systems on K3 surfaces is of primary importance.

	\begin{proposition}[\cite{72768972-f7b8-3f72-b8d3-5e64ee0f3017}, Corollary 3.2]\label{no fixed components}
	Let $|D|$ be a complete linear system on a K3 surface. Then $|D|$ has no base points outside its fixed components.
	\end{proposition}

	\begin{definition}[\cite{72768972-f7b8-3f72-b8d3-5e64ee0f3017}, Definition 3.3]
		Let $D$ be an effective divisor on a nonsingular surface $X$. We say that $D$ is \emph{numerically $m$-connected} if for any decomposition $D = D_1 + D_2$ with $D_i > 0$ the inequality $D_1 \cdot D_2 \geq m$ holds.
	\end{definition}

	\begin{lemma}[\cite{72768972-f7b8-3f72-b8d3-5e64ee0f3017}, Lemma 3.7]\label{numerically 2-connected}
		Let $C$ be an irreducible curve on a K3 surface $X$ such that $C^2 > 0$. Then any divisor in the linear system $|C|$ is numerically $2$-connected.
	\end{lemma}

	\begin{remark}\label{2-connectedness of a polarisation}
		If $(X, H)$ is a quasi-polarized K3 surface then $H^2 > 0$.
		In general, the linear system~$|H|$ may have fixed components (see~\cite[Proposition~8.1]{72768972-f7b8-3f72-b8d3-5e64ee0f3017});
		however, if~$(S,H)$ is Brill--Noether general, $|H|$ has no fixed components by~\cite[Remark~2.9]{bayer2024},
		hence by Proposition~\ref{no fixed components} one has $\Bs{|H|} = \varnothing$,
		and therefore Bertini's theorem implies that a general element in $|H|$ is irreducible
		and by Lemma~\ref{numerically 2-connected} any divisor in $|H|$ is numerically 2-connected.

	\end{remark}

	Now we prove a simple technical lemma that will be used later.

	\begin{lemma}\label{intersection of two divisors without fixed parts}
		Let $M_1$ and $M_2$ be divisors without fixed components.
		\begin{enumerate}[label={\textup{(\arabic*)}}]
			\item \label{case 1 lemma intersection of two divisors without fixed parts} If $M_1 \cdot M_2 = 0$,
			then $M_1 \sim n_1E$ and $M_2 \sim n_2E$ where $E$ is a smooth curve of genus $1$ and $n_1, n_2 \geq 1$.
			In other words, the linear systems $|M_1|$ and $|M_2|$ are composed with the same elliptic pencil $|E|$.
			\item \label{case 2 lemma intersection of two divisors without fixed parts} If $M_1 \cdot M_2 > 0$, then $(M_1 + M_2)^2 \geq 4$.
		\end{enumerate}
	\end{lemma}

	\begin{proof}
		By assumption $M_i$ have no fixed components, so $M_i^2 \geq 0$ and $(M_1 + M_2)^2 \geq 0$.

		 Suppose that $M_1 \cdot M_2 > 0$. In this case $(M_1 + M_2)^2 > 0$.
		 Moreover, the linear system $|M_1 + M_2|$ is base-point-free; hence, a general curve in $|M_1 + M_2|$ is irreducible by the theorem of Bertini,
		 and therefore it follows from Lemma~\ref{numerically 2-connected} that the divisor $M_1 + M_2$ is numerically 2-connected, i.e., $M_1 \cdot M_2 \geq 2$; we conclude that $(M_1 + M_2)^2 \geq 2M_1 \cdot M_2 \geq 4$.
		 This proves~\ref{case 2 lemma intersection of two divisors without fixed parts}.

		Now suppose that $M_1 \cdot M_2 = 0$. Note that in this case we must have $(M_1 + M_2)^2 = 0$. Indeed, if $(M_1 + M_2)^2 > 0$ then the divisor $M_1 + M_2$ is numerically 2-connected by Lemma~\ref{numerically 2-connected}, which contradicts the assumption.
		Hence, $M_1^2 = M_2^2 = 0$ and it follows from~\cite[Proposition 2.6]{72768972-f7b8-3f72-b8d3-5e64ee0f3017}
		that $M_1 \sim n_1E_1$ and $M_2 \sim n_2E_2$ where $E_i$ are smooth curves of genus 1 and $n_i \geq 1$. Finally, since $M_1 \cdot M_2 = 0$, we conclude that $E_1 \sim E_2$.
	\end{proof}

	The following elementary observation is also quite useful.

	\begin{lemma}
	\label{lem:h0-chi}
	If~$D > 0$ then~$h^0(D) \ge \chi(D)$.
	\end{lemma}

	\begin{proof}
	It follows from Serre duality that~$h^2(D) = h^0(-D)$, and since~$D > 0$, we have~$h^0(-D) = 0$.
	Therefore, we obtain~$\chi(D) = h^0(D) - h^1(D) \le h^0(D)$.
	\end{proof}

	The next lemma will be crucial for the proof of Proposition~\ref{conditions on a polarizing divisor} and for the proof of the main theorem.

	\begin{lemma}\label{sorting fixed components}
		Let~$(X,H)$ be a quasi-polarized K3 surface.
		Suppose that there is a decomposition
		\begin{equation*}
		H = D_1 + \dots + D_n + \Delta,
		\end{equation*}
		where~$n > 0$, $D_i > 0$ and~$\Delta \geq 0$. Then there is a decomposition
		\begin{equation*}
		H = \bar{D}_1 + \dots + \bar{D}_n
		\end{equation*}
		with~$\bar{D}_i > 0$ and~$\bar{D}_i^2 \geq D_i^2$.
		Moreover, if $h^0(D_i) \geq 2$ we can arrange $\bar{D}_i$ to satisfy $\bar{D}_i^2 \geq 0$.
	\end{lemma}

	\begin{proof}
		To prove the first part of the proposition we write $\Delta = \Delta_1 + \dots + \Delta_m$, where all $\Delta_i$ are irreducible curves and apply induction on $m$. If $m = 0$ then the decomposition $H = D_1 + \dots + D_n$ is the desired one.
		Suppose $m > 0$. Since $H$ is numerically $1$-connected (Remark~\ref{2-connectedness of a polarisation}), we have
		\begin{equation*}
		(D_1 + \dots + D_n) \cdot (\Delta_1 + \dots + \Delta_m) \geq 1.
		\end{equation*}
		Therefore, one can find $i$ and $j$ such that $D_i \cdot \Delta_j\geq 1$.
		Set $\widehat{D}_i = D_i + \Delta_j$ and note that
		\begin{equation*}
		\widehat{D}_i^2 = (D_i + \Delta_j)^2 \geq D_i^2 + 2D_i\cdot \Delta_j - 2 \geq D_i^2,
		\end{equation*}
		where the first inequality holds because $\Delta_j$ is an irreducible curve on a K3 surface,
		hence, $\Delta_j^2 \geq -2$ by \cite[(2.4)]{72768972-f7b8-3f72-b8d3-5e64ee0f3017}.
		Furthermore, for $k \neq i$  set $\widehat{D}_k = D_k$.
		Then the induction hypothesis applies to the decomposition
		$H = \widehat{D}_1 + \widehat{D}_2 + \dots + \widehat{D}_n  + (\Delta - \Delta_j)$.

		For the last part of the lemma,
		if $h^0(D_i) \geq 2$, we write $D_i = M_i + F_i$, where $F_i$ is the union of the fixed components of~$D_i$
		and~$M_i = D_i - F_i$ so that~$M_i^2 \ge 0$ and~$h^0(M_i) = h^0(D_i)$, in particular, $M_i > 0$.
		Then, if $D_i^2 < 0$, we take $D_i' = M_i$ and $\Delta' = \Delta + F_i$ and, if $D_i^2 \ge 0$, we take $D_i' = D_i$ and $\Delta' = \Delta$ and apply the above argument
		to the decomposition $H = D_1 + \dots + D_i' + \dots + D_n + \Delta'$.
		Note that one has $D_i'^2 \geq D_i^2$ and $D_i'^2 \geq 0$, so the conditions of the lemma are satisfied.
	\end{proof}

	\subsection{Necessary conditions for Brill--Noether generality}\label{subsec The necessary conditions for Brill--Noether generality}

	The following proposition is the main result of this section.
	
	\begin{proposition}\label{conditions on a polarizing divisor}
		Let $(X,H)$ be a quasi-polarized K3 surface.
		Assume that there is a decomposition
		\begin{equation*}
		H = D_1 + \dots + D_n,
		\end{equation*}
		where $h^0(D_i) \geq 2$ for all $1 \leq i \leq n$ and either of the following conditions is satisfied:
		\begin{enumerate}[label={\textup{(\arabic*)}}]
			\item \label{case1} $n \geq 5$;
			\item \label{case2} $n = 4$, $D_1^2 \ge 2$ and~$D_2^2, D_3^2, D_4^2 \ge 0$;
			\item \label{case3} $n = 3$ and one of the conditions holds:
			\begin{enumerate}[label={\textup{(\alph*)}}]
				\item \label{case3a 222} $D_1^2 \geq D_2^2 \geq D_3^2 \ge 2$;
				\item \label{case3b 521} $D_1^2 \geq 8$, $D_2^2 = 2$, $D_3^2 = 0$;
				\item \label{case3c 331} $D_1^2 \geq D_2^2 \geq 4$, $D_3^2 = 0$.
			\end{enumerate}
		\end{enumerate} Then the K3 surface~$(X,H)$ is not Brill--Noether general.
	\end{proposition}
	
	\begin{remark}\label{exceptional cases}
		One can also reformulate Proposition~\ref{conditions on a polarizing divisor} as necessary conditions for the Brill--Noether property.
		Namely, if the K3 surface~$(X,H)$ is Brill--Noether general and $H = D_1 + \dots + D_n$, where $n \geq 3$, $D_i > 0$, and $D_i^2 \geq 0$,
		then one of the following holds:

		\begin{enumerate}[label={\textup{(\arabic*)}}]
			\item
			\label{case:bn-4}
			$n=4$ and $D_1^2 = D_2^2 = D_3^2 = D_4^2 =0$;
			\item
			\label{case:bn-3}
			$n = 3$ and up to renumbering of the $D_i$, one of the following conditions holds:
			\begin{enumerate}[label={\textup{(\alph*)}}]
				\item $D_1^2 \in \{ 2,4,6\}$, $D_2^2 = 2$, $D_3 ^2 = 0$;
				\item $D_2^2 = 0$, $D_3 ^2 = 0$.
			\end{enumerate}
		\end{enumerate}
	\end{remark}

	Before proving the proposition, we prove few auxiliary lemmas.
	We start with a computation on which most of the arguments of this section rely.

	\begin{lemma}\label{no negative intersections}
		Let $(X,H)$ be a quasi-polarized K3 surface and $H = D_1 + D_2 + D_3$, where $D_i>0$ and $D_i^2 \geq 0$.
		Suppose that one of the following two conditions holds:
		\begin{enumerate}[label={\textup{(\arabic*)}}]
		\item
		\label{it:condition-1}
		$(D_1 \cdot D_2) \left( \frac{1}{2}D_3^2 + 1 \right) \ge D_1 \cdot D_3 + D_2 \cdot D_3$, or
		\item
		\label{it:condition-2}
		$D_2 \cdot D_3 \leq 2 + \frac{1}{2}\left(D_1^2 + D_2^2 + D_3^2\right)$.
		\end{enumerate}
		Then~$\chi \left( D_1 + D_2 \right)\chi\left(D_3\right) > g$;
		in particular, the K3 surface $(X,H)$ is not Brill--Noether general.
	\end{lemma}
	
	\begin{proof}
		We have~$\chi(D) = \frac{1}{2}D^2 + 2$ and $g = \frac{1}{2}H^2 + 1$;
		hence,~$\chi \left( D_1 + D_2 \right)\chi\left(D_3\right) - g$ can be rewritten as
		\begin{align*}
			&\left(\frac{1}{2}\left(D_1 + D_2\right)^2 + 2\right)\left(\frac{1}{2}D_3^2 + 2\right) - \frac{1}{2}(D_1 + D_2 +  D_3)^2 - 1
			\\
			&\qquad= \left(\frac{1}{2}\left(D_1 + D_2\right)^2 + 1 \right)\left(\frac{1}{2}D_3^2 + 1 \right) + \left(\frac{1}{2}D_3^2 + 1 \right) + \left(\frac{1}{2}\left(D_1 + D_2\right)^2 + 1\right) + 1
			\\
			&\qquad- \frac{1}{2}\left(\left(D_1 + D_2\right)^2 + D_3^2\right) - (D_1 + D_2) \cdot D_3 - 1
			\\
			&\qquad= \left(\frac{1}{2}(D_1 + D_2)^2 + 1\right)\left(\frac{1}{2}D_3^2 + 1 \right)
			+ 2 - D_1 \cdot D_3 -  D_2 \cdot D_3
			\\
			&\qquad= \frac{1}{4}\left(D_1^2 + D_2^2\right)D_3^2 + \frac{1}{2}\left(D_1^2 + D_2^2 + D_3^2\right)
			+ (D_1 \cdot D_2)\left(\frac{1}{2}D_3^2 + 1 \right) + 3 - D_1 \cdot D_3 -  D_2 \cdot D_3.
		\end{align*}

		The first and the second summands in the right-hand side are non-negative.
		Therefore, if~\ref{it:condition-1} is satisfied, the right-hand side is positive;
		hence,~$\chi \left( D_1 + D_2 \right)\chi\left(D_3\right) > g$.

		Assume~\ref{it:condition-2} holds.
		Renumbering~$D_2$ and~$D_3$ if necessary, we can assume that $D_1 \cdot D_2 \geq D_1 \cdot D_3$.
		Since by Remark~\ref{2-connectedness of a polarisation} every divisor in $|H|$  is numerically 2-connected,
		we have $(D_2 + D_3) \cdot D_1 \geq 2$ and it follows that~$D_1 \cdot D_2 \ge 1$; hence,
		\begin{equation*}
		(D_1 \cdot D_2)\left(\frac{1}{2}D_3^2 + 1 \right) - D_1 \cdot D_3 \ge D_1 \cdot D_2 - D_1 \cdot D_3 \ge 0.
		\end{equation*}
		Since also we have $\frac{1}{2}(D_1^2 + D_2^2 + D_3^2) + 2 - D_2 \cdot D_3 \ge 0$ by~\ref{it:condition-2},
		the right-hand side above is positive again; hence,~$\chi \left( D_1 + D_2 \right)\chi\left(D_3\right) > g$.
		
		Finally, using Lemma~\ref{lem:h0-chi} we conclude that~$h^0(D_1 + D_2)h^0(D_3) \geq \chi(D_1 + D_2)\chi(D_3) > g$
		and therefore the K3 surface $(X,H)$ is not Brill--Noether general.
	\end{proof}

	Now we deduce several corollaries of Lemma~\ref{no negative intersections}.

	\begin{lemma}\label{lemma for n = 4}
		Let $(X,H)$ be a quasi-polarized K3 surface and $H = D_1 + D_2 + D_3 + D_4$, where~$D_i > 0$.
		Suppose that the following two conditions hold:
		\begin{enumerate}[label={\textup{(\arabic*)}}]
			\item \label{lemma for n=4 condition 1} $D_1^2, D_2^2, D_3^2, D_4^2 \geq 0$;
			\item \label{lemma for n=4 condition 2} $D_2 \cdot D_3 \leq D_2 \cdot D_4 \leq D_3 \cdot D_4 \leq 1$.
		\end{enumerate}
		Then the K3 surface $(X,H)$ is not Brill--Noether general.
	\end{lemma}

	\begin{proof}
		By 2-connectedness of $H$ we have $D_4 \cdot (D_1 + D_2 + D_3) \geq 2$.
		Then it follows from the assumption~\ref{lemma for n=4 condition 2} that $D_1 \cdot D_4 \geq 0$.
		We conclude that~$\left(D_1 + D_4\right)^2 \geq 0$
		and $D_2 \cdot D_3 \leq 1 < 2 + \frac{1}{2}\left(\left(D_1 + D_4\right)^2 + D_3^2 + D_2^2\right)$.
		Hence, Lemma~\ref{no negative intersections}\ref{it:condition-2}
		applies to the decomposition $H = (D_1 + D_4) + D_2 + D_3$ and implies that the K3 surface~$(X,H)$ is not Brill--Noether general.
	\end{proof}

	The next result applies in situations where we have a decomposition of~$H$ with many elliptic summands.

	\begin{lemma}\label{multiple elliptic curves}
		Let $(X,H)$ be a quasi-polarized K3 surface and suppose that $H = nE + D$, where $D>0$ and $E$ is an elliptic curve.
		Suppose that one of the following conditions holds
		\begin{enumerate}[label={\textup{(\arabic*)}}]
			\item \label{multiple elliptic curves condition 1} $n \geq 2$ and $D^2 \geq 0$;
			\item \label{multiple elliptic curves condition 2} $n \geq 4$.
		\end{enumerate}
		Then the K3 surface $(X,H)$ is not Brill--Noether general.
	\end{lemma}
	
	\begin{proof}
		If $n \geq 4$ then $H - 4E \geq 0$ and, since $H$ is numerically effective,	we have
		\begin{equation*}
		\left(H-2E\right)^2 = H^2 - 4H \cdot E = H \cdot(H-4E) \geq 0.
		\end{equation*}
		Hence, case~\ref{multiple elliptic curves condition 2} is a subcase of~\ref{multiple elliptic curves condition 1} with $D = H - 2E$.
		
		In case~\ref{multiple elliptic curves condition 1} we have $h^0(E) = 2$ and
		\begin{equation*}
		h^0 \left( (n-1)E + D \right) \geq \chi \left( (n-1)E + D \right) =
		\frac{1}{2} \left( (n-1)E + D \right)^2 + 2 = \frac{1}{2}D^2 + (n-1)E \cdot D + 2.
		\end{equation*}
		On the other hand, we have~$g = \frac{1}{2}(nE + D)^2 + 1 = \frac{1}{2}D^2 + nE \cdot D + 1$. Hence,
		\begin{equation*}
		h^0(E)h^0 \left( (n-1)E+D \right) \geq
		2\chi \left( (n-1)E+D \right) =
		2\left(\frac{1}{2}D^2 + \left(n-1\right)E \cdot D + 2\right) > \frac{1}{2}D^2 + nE\cdot D + 1 = g,
		\end{equation*}
		where we used that $D^2 \geq \frac{1}{2}D^2$, $2(n-1) \geq n$ and $E \cdot D \geq 0$ because $|E|$ has no fixed components.
		The resulting inequality shows that the K3 surface~$(X,H)$ is not Brill--Noether general.
	\end{proof}

	We are now ready to prove Proposition~\ref{conditions on a polarizing divisor}.

	\begin{proof}[Proof of Proposition~\textup{\ref{conditions on a polarizing divisor}}]
	We consider all cases of the proposition starting from the last one.

	{\bf Case~\ref{case3}.}
		By Lemma~\ref{no negative intersections}\ref{it:condition-2}
		we can assume that the inequality $D_i \cdot D_j > 0$ holds for all $i \neq j$.
		To prove the proposition it is enough to show that in either of the subcases~\ref{case3a 222}, \ref{case3b 521}, or~\ref{case3c 331} we have
		\begin{equation}
		\label{eq:chi-h-di}
		\chi(H - D_i)\chi(D_i) - g > 0
		\end{equation}
		for some $i \in \{1,2,3\}$.
		
		Suppose for a contradiction that~\eqref{eq:chi-h-di} is not true for all~$i$.
		Denote $a_i = \frac{1}{2}D_i^2 + 1$ and~$x_{ij} = D_i \cdot D_j$. Then, since~$D_i^2 \geq 0$, it follows from
		Lemma~\ref{no negative intersections}\ref{it:condition-1} that the following inequalities hold simultaneously:
		\begin{align*}
			& a_3 x_{12} =\left(\frac{1}{2}D_3^2 + 1\right)\left(D_1\cdot D_2\right) < D_1 \cdot D_3 + D_2 \cdot D_3 = x_{13} + x_{23}, \\
			& a_2 x_{13} = \left(\frac{1}{2}D_2^2 + 1\right)\left(D_1\cdot D_3\right) < D_1 \cdot D_2 + D_2 \cdot D_3 = x_{12} + x_{23}, \\
			& a_1 x_{23} = \left(\frac{1}{2}D_1^2 + 1\right)\left(D_2\cdot D_3\right) < D_1 \cdot D_2 + D_1 \cdot D_3 = x_{12} + x_{13},
		\end{align*}
		Since~$x_{ij} > 0$, it follows that
		\begin{equation*}
		\begin{pmatrix}
		x_{12} & x_{13} & x_{23}
		\end{pmatrix}
		\begin{pmatrix}
		a_3 & -1 & -1 \\ -1 & a_2 & -1 \\ -1 & -1 & a_1
		\end{pmatrix}
		\begin{pmatrix}
		x_{12} \\ x_{13} \\ x_{23}
		\end{pmatrix}
		< 0,
		\end{equation*}
		i.e., the matrix in the middle cannot be positive semidefinite.
		Since on the other hand, its diagonal entries~$a_i$ are positive,
		and its principal 2-by-2 minors~$a_ia_j - 1$ are nonnegative,
		its determinant must be negative, i.e.,
		\begin{equation}
		\label{eq:determinant}
		a_1a_2a_3 - a_1 - a_2 - a_3 - 2 < 0.
 		\end{equation}

		Now, in case~\ref{case3}\ref{case3a 222}, we have~$a_1 \ge a_2 \ge a_3 \ge 2$; hence,~$a_2a_3 \ge 4$ and
		\begin{equation*}
		a_1a_2a_3 - a_1 - a_2 - a_3 - 2 \ge 4a_1 - a_1 - a_2 - a_3 - 2 = (a_1 - a_2) + (a_1 - a_3) + (a_1 - 2) \ge 0
		\end{equation*}
		in contradiction to~\eqref{eq:determinant}.
		Similarly, in case~\ref{case3}\ref{case3b 521}, we have~$a_1 \ge 5$, $a_2 = 2$, $a_3 = 1$; hence,
		\begin{equation*}
		a_1a_2a_3 - a_1 - a_2 - a_3 - 2 = 2a_1 - a_1 - 5 = a_1 - 5 \ge 0
		\end{equation*}
		and in case~\ref{case3}\ref{case3c 331}, we have~$a_1, a_2 \ge 3$, $a_3 = 1$; hence,
		\begin{equation*}
		a_1a_2a_3 - a_1 - a_2 - a_3 - 2 = a_1a_2 - a_1 - a_2 - 3 = (a_1 - 1)(a_2 - 1) - 4 \ge 0,
		\end{equation*}
		again in contradiction to~\eqref{eq:determinant}.
		Therefore, one of the inequalities~\eqref{eq:chi-h-di} must hold, so that the K3 surface~$(X,H)$ is not Brill--Noether general by Lemma~\ref{lem:h0-chi}.
		
	{\bf Case~\ref{case2}.} Renumbering the $D_i$ if necessary, we may assume that
		\begin{equation*}
		D_1^2 \geq 2,
		\qquad
		D_2^2, D_3^2, D_4^2 \geq 0,
		\qquad\text{and}\qquad
		D_2 \cdot D_3 \leq D_2 \cdot D_4 \leq D_3 \cdot D_4.
		\end{equation*}

		By Lemma~\ref{lemma for n = 4} it is enough to consider the case
		where $D_3 \cdot D_4 \geq 2$; in particular, $(D_3 + D_4)^2 \geq 0$.

		First, suppose that $D_2 \cdot D_3 \leq D_2 \cdot D_4 \leq D_3 \cdot D_4 \leq 3$.
		Then we have
		
		\begin{equation*}
		D_2 \cdot (D_3 + D_4) \leq 3 + D_3 \cdot D_4 \leq
		2 + \frac{1}{2}\left(D_1^2 +D_2^2 + \left(D_3 + D_4\right)^2\right),
		\end{equation*}
		where the first inequality follows from the assumptions $D_2 \cdot D_3 \leq 3$ and~$D_2 \cdot D_4 \le D_3 \cdot D_4$,
		and the second inequality follows from inequalities $D_1^2 \geq 2$,~$D_2^2 \ge 0$
		and~$\frac12(D_3 + D_4)^2 \geq D_3 \cdot D_4$.
		Since $D_1^2 \ge 0$, $D_2^2 \ge 0$ and~$(D_3 + D_4)^2 \geq 0$, Lemma~\ref{no negative intersections}\ref{it:condition-2} applies to the decomposition $H = D_1 + D_2 + (D_3 + D_4)$
		and we conclude that the K3 surface $(X,H)$ is not Brill--Noether general.

		Now, assume that $D_3 \cdot D_4 \geq 4$, hence~$(D_3 + D_4)^2 \geq 8$.
		Then, since~$D_1^2 \geq 2$ by assumption,
		we can apply~\ref{case3}\ref{case3b 521} of the proposition (proved above)
		to the decomposition $H = (D_3 + D_4) + D_1 + D_2$
		and it follows that the K3 surface $(X,H)$ is not Brill--Noether general.

	{\bf Case~\ref{case1}.}
		Since $h^0(D_i) \geq 2$, we can write $D_i = M_i + F_i$, where $M_i > 0$ has no fixed components,
		so that in particular~$M_i^2 \ge 0$ and $F_i \geq 0$ is a fixed divisor. Set $\Delta = F_1 + \dots + F_n$.

		We will show that we can renumber the~$M_i$ in such a way that~$(M_1 + M_2)^2 \ge 4$ and~$(M_3+M_4)^2 \ge 4$
		and thus reduce to case~\ref{case3}\ref{case3c 331} of the proposition.
		By Lemma~\ref{intersection of two divisors without fixed parts},
		we need to make sure that~$M_1 \cdot M_2 \ne 0$ and~$M_3 \cdot M_4 \ne 0$.
		To this end, we define an equivalence relation on the set of the~$M_i$ by saying that $M_i$ is equivalent to $M_j$ if $M_i \cdot M_j = 0$;
		this is indeed an equivalence relation by Lemma~\ref{intersection of two divisors without fixed parts}.
		
		First, assume that one equivalence class of the $M_i$ contains at least four elements.
		Then, again by Lemma~\ref{intersection of two divisors without fixed parts},
		the corresponding four divisor classes are multiplies of the same elliptic curve $E$.
		Therefore, $H = nE + D$ with $n \geq 4$ and the K3 surface $(X,H)$ is not Brill--Noether general by Lemma~\ref{multiple elliptic curves}.
		
		Next, assume that every equivalence class has cardinality at most three.
		Then we can find two pairs of non-equivalent divisors.
		Indeed, if all the classes have cardinality one, we can take two arbitrary pairs.
		Otherwise, we take the first elements in each pair from an equivalence class with cardinality at least two
		and the second elements from its complement (which also has cardinality at least two).

		Renumbering, if necessary, we may assume the first pair is $M_1$, $M_2$ and the second pair is $M_3$, $M_4$,
		so that~$M_1 \cdot M_2 > 0$ and~$M_3 \cdot M_4 > 0$.
		Now, Lemma~\ref{intersection of two divisors without fixed parts} shows that $(M_1 + M_2)^2 \geq 4$ and $(M_3 + M_4)^2 \geq 4$.
		Therefore, applying Lemma~\ref{sorting fixed components} to the decomposition
		\begin{equation*}
		H = (M_1 + M_2) + (M_3 + M_4) + (M_5 + \dots + M_n) + \Delta
		\end{equation*}
		we obtain a decomposition $H = \bar{D}_1 + \bar{D}_2 + \bar{D}_3$ with $\bar{D}_1^2 \geq (M_1 + M_2)^2 \geq 4$, $\bar{D}_2^2 \geq (M_3 + M_4)^2 \geq 4$
		and $\bar{D}_3^2 \geq (M_5 + \dots + M_n)^2 \geq 0$.
		By case~\ref{case3}\ref{case3c 331} of the proposition (proved above)
		we conclude that the K3 surface $(X,H)$ is not Brill--Noether general.
	\end{proof}

	\section{Harmonic filtrations of Lazarsfeld bundles}\label{sec Lazarsfeld Bundles}

	In this section we use endomorphisms of non-simple Lazarsfeld bundles
	to construct a decomposition of the quasi-polarization~$H$ into a sum of effective divisor classes.
	We start by introducing a class of vector bundles on K3 surfaces containing and generalizing Lazarsfeld bundles.

	\begin{definition}
	\label{def:good-sheaf}
	We will say that a nonzero vector bundle~$F$ on a K3 surface~$X$ is \emph{balanced} if
		\begin{enumerate}[label={\textup{(\alph*)}}]
			\item
			\label{it:balanced-hi}
			$h^0(F) = h^1(F) = 0$ and
			\item
			\label{it:balanced-gg}
			$F^{\vee}$ is globally generated in codimension $1$.
		\end{enumerate}
	Note that property~\ref{it:balanced-gg} is equivalent to the surjectivity of the evaluation morphism
	\begin{equation}
	\label{eq:ev-f}
	H^0(F^\vee) \otimes \mathcal{O}_X \to F^\vee
	\end{equation}
	on a complement of a finite set.
	\end{definition}
	
	Recall from Section~\ref{sec Preliminaries} that for a sheaf $F$ on a K3 surface
	we denote by $v(F) = \langle r(F) , c_1(F), s(F) \rangle$ its Mukai vector,
	where $r(F) = \Rk{F}$ and $s(F) = ch_2(F) + r(F) = \chi(F) - r(F)$.

	\begin{lemma}\label{Mukai vector lemma}
		If $F$ is a balanced vector bundle on a K3 surface then $s(F) \geq 1$.
		Moreover, the dual of the evaluation morphism~$F \to H^0(F^\vee)^\vee \otimes \mathcal{O}_X$ is a monomorphism.
	\end{lemma}

	\begin{proof}
		Let $v(F) = \langle r, -D, s \rangle$ be the Mukai vector of $F$.
		Since $h^0(F) = h^1(F) = 0$ by definition, we have $h^0(F^{\vee}) = h^2(F) = \chi(F) = r + s$.
		Therefore, the evaluation morphism~\eqref{eq:ev-f} takes the form~\mbox{$\mathcal{O}_X^{\oplus(r + s)} \longrightarrow F^{\vee}$};
		since by definition it is surjective in codimension 1, it follows that $s \geq 0$.

		Suppose that $s = 0$. Then we obtain an exact sequence
		\begin{equation}\label{qwe1}
			0\longrightarrow \mathcal{O}_X^{\oplus r} \longrightarrow F^{\vee} \longrightarrow B \longrightarrow 0,
		\end{equation}
		where $\dim\Supp{B} = 0$. Dualizing~\eqref{qwe1}, we see that $F^{\vee \vee} \cong \mathcal{O}_X^{\oplus r}$.
		But $F$ is reflexive; hence, $F \cong \mathcal{O}_X^{\oplus r}$ and it follows that~$h^0(F) = r> 0$, a contradiction.
		Thus, $s \geq 1$, as desired.

		Finally, since the evaluation morphism~\eqref{eq:ev-f} is generically surjective,
		its dual~$F \to H^0(F^\vee)^\vee \otimes \mathcal{O}_X$ is generically injective,
		and since~$F$ is torsion free, it is a monomorphism.
	\end{proof}

	Recall that a subsheaf~$G \subset F$ is called \emph{saturated} if the quotient~$F/G$ is torsion free
	(see, e.g., \cite[Definition~1.1.5]{huybrechts2010geometry}).
	Recall also that for every subsheaf~$G \subset F$, if~$T(F/G) \subset F/G$ is the torsion subsheaf of~$F/G$,
	then the \emph{saturation} of~$G$ in~$F$ is the sheaf
	\begin{equation}
	\label{eq:saturation}
	\widehat{G} \coloneqq \Ker(F \to (F/G)/T(F/G)).
	\end{equation}
	This is the unique saturated subsheaf of~$F$ containing~$G$ and such that~$\Rk{\widehat{G}} = \Rk{G}$.
	Note also that the definition~\eqref{eq:saturation} of~$\widehat{G}$
	implies an exact sequence~$0 \longrightarrow G \longrightarrow \widehat{G} \longrightarrow T(F/G) \longrightarrow 0$,
	and since the first Chern class of a torsion sheaf is always nonnegative, it follows that
	\begin{equation}
	\label{eq:saturation-c1}
	c_1(\widehat{G}) \ge c_1(G).
	\end{equation}
	Finally, recall that if~$F$ is a locally free sheaf on a smooth surface and~$G \subset F$ is saturated then~$G$ is locally free;
	in particular, the saturation of any subsheaf in~$F$ is locally free.

	\begin{lemma}\label{l0}
		Let $G \subset F$ be a nonzero subsheaf in a balanced vector bundle and let $\widehat{G}$ be its saturation.
		Denote~$D \coloneqq -c_1(G)$ and $\widehat{D} \coloneqq -c_1(\widehat{G})$. Then
		\begin{enumerate}[label={\textup{(\arabic*)}}]
			\item
			\label{it:case-zero-h0}
			$h^0(G) = h^0(\widehat{G}) = 0$;
			\item
			\label{it:case-global-generation}
			$\widehat{G}^\vee$ is globally generated away from a finite set;
			\item
			\label{it:h0}
			$D \geq \widehat{D} > 0$, the linear system $|\widehat{D}|$ is base-point-free and $h^0(D) \geq h^0(\widehat{D}) \geq 2$.
		\end{enumerate}
	\end{lemma}

	\begin{proof}
		We have~$G \subset \widehat{G} \subset F$ by definition of~$G$ and~$\widehat{G}$
		and~$h^0(F) = 0$ by definition of a balanced bundle, so~\ref{it:case-zero-h0} follows immediately.

		Denote $E = F/\widehat{G}$, it is a torsion free sheaf and there is an exact sequence:
		\begin{equation*}
			0 \longrightarrow \widehat{G} \longrightarrow F \longrightarrow E \longrightarrow 0.
		\end{equation*}
		Dualazing it, we obtain an exact sequence
		\begin{equation*}
			0 \longrightarrow E^{\vee} \longrightarrow F^{\vee} \longrightarrow \widehat{G}^{\vee} \longrightarrow
			\mathcal{E}xt^1(E, \mathcal{O}_X) \longrightarrow 0,
		\end{equation*}
		and $Z :=\Supp(\mathcal{E}xt^1(E, \mathcal{O}_X))$ is a finite set because $E$ is torsion free.
		Since~$F^{\vee}$ by definition is globally generated away from a finite set, say, $Z_F$,
		we conclude that~$\widehat{G}^{\vee}$ is globally generated away from $Z \bigcup Z_F$, which is also a finite set.
		This proves~\ref{it:case-global-generation}.

		To prove~\ref{it:h0}, note that the map~$H^0(\widehat{G}^{\vee}) \otimes \mathcal{O}_X \longrightarrow G^{\vee}$
		is surjective in codimension 1 (by part~\ref{it:case-global-generation}),
		hence the same is true for its exterior power
		\begin{equation}\label{qwe111}
			\bigwedge\nolimits^m H^0(\widehat{G}^{\vee}) \otimes \mathcal{O}_X \longrightarrow \Det{\widehat{G}^{\vee}} \cong \mathcal{O}(\widehat{D}),
		\end{equation}
		where~$m = \Rk \widehat{G}$.
		Thus, the line bundle~$\mathcal{O}(\widehat{D})$ is also globally generated away from a finite set;
		in particular, $\widehat{D} \geq 0$, and the linear system~$|\widehat{D}|$ is base-point-free by Proposition~\ref{no fixed components}.

		To prove that~$\widehat{D} > 0$, note that~$\widehat{G} \subset F \subset H^0(F^\vee)^\vee \otimes \mathcal{O}_X$
		by Lemma~\ref{Mukai vector lemma}.
		Therefore, if~$\widehat{D} = 0$ then~$\widehat{G} \cong \mathcal{O}_X^{\oplus m}$ by Lemma~\ref{lem:cc},
		in contradiction to~$h^0(\widehat{G}) = 0$.

		Finally, as the linear system $|\widehat{D}|$ has no fixed components and $\widehat{D} > 0$,
		we have $h^0(\widehat{D}) \geq 2$.
		On the other hand, \eqref{eq:saturation-c1} implies~$-\widehat{D} \ge -D$, hence~$D \ge \widehat{D}$,
		and we conclude that~$h^0(D) \geq h^0(\widehat{D})$.
	\end{proof}

	\begin{lemma}\label{Mukai vector of a kernel}
		Let~$F$ be a balanced vector bundle.
		If~$G \subset F$ is a subsheaf such that~$F/G$ is torsion free and~$h^0(F/G) = 0$
		then~$G$ is a balanced vector bundle.
	\end{lemma}

	\begin{proof}
		The vanishing~$h^0(G) = 0$ is proved in Lemma~\ref{l0}\ref{it:case-zero-h0}
		and~$h^1(G) = 0$ follows from~$h^1(F) = 0$ (by definition of a balanced bundle) and~$h^0(F/G) = 0$ (by assumption).
		Finally, $G = \widehat{G}$ is globally generated away from a finite set by Lemma~\ref{l0}\ref{it:case-global-generation}.
	\end{proof}	
	
	Now we axiomatize the properties of the filtration that we are going to construct.
	
	\begin{definition}
	\label{def:good-filtration}
		Let $F$ be a vector bundle and let $F = F_1 \supsetneq F_2 \supsetneq \dots \supsetneq F_{n} \supsetneq F_{n+1} = 0$ be its filtration.
			We will say that a filtration $F_{\bullet}$ is \emph{harmonic} if for all $1 \leq i \leq n$ one has
		\begin{enumerate}[label={\textup{(\arabic*)}}]
			\item
			\label{filt:gi}
			$F_i/F_{i+1}$ is a simple torsion free sheaf
			and if~$i \neq n$ then~$\Rk{(F_i/F_{i+1})} \leq \frac{1}{2}\Rk{F_i}$;
			\item
			\label{filt:cgi}
			if $D_i := -c_1(F_i/F_{i+1})$ then $D_i > 0$ and $h^0(D_i) \geq 2$.
		\end{enumerate}
	\end{definition}
	
	To construct a harmonic filtration on balanced bundles we use the following lemma.
	
	\begin{lemma}\label{simple image rk leq 1/2}
		Let $F$ be a torsion free sheaf and suppose that $F$ has a nonscalar endomorphism. Then $F$ has a $($nonzero$)$ endomorphism $\phi$ such that the sheaf $\Image{\phi}$ is simple and $\Rk{\Image{\phi}} \leq \frac{1}{2}\Rk{F}$.
	\end{lemma}
	
	\begin{proof}
		We argue by induction on $\Rk{F}$. First, note that there is a (nonzero) endomorphism $\psi : F\longrightarrow F$ such that $\Rk{\psi} \leq \frac{1}{2}\Rk{F}$. Indeed, if $F$ is decomposable, i.e., $F = F_1 \bigoplus F_2$ and $F_i \neq 0$,
		then without loss of generality we can assume that $\Rk{F_1} \leq \frac{1}{2}\Rk{F}$.
		In this case take $\psi$ to be the projection to $F_1$.
		If $F$ is indecomposable then it necessarily has a (nonzero) nilpotent endomorphism $\bar{\psi}$.
		Consider $k$ such that~$\bar{\psi}^k =0$ and $\bar{\psi}^{k-1} \neq 0$,
		so that $\Image{\bar{\psi}^{k-1}} \hookrightarrow \Ker{\bar{\psi}^{k-1}}$.
		Then~$\Rk{(\Image{\bar{\psi}^{k-1}})} \leq \frac{1}{2}\Rk{F}$ and, therefore, we can take $\psi = \bar{\psi}^{k-1}$.
		
		Now let~$\psi$ be as above.
		If $F' \coloneqq \Image{\psi}$ is simple, we set $\phi = \psi$.
		Otherwise, we note that~$F'$ is torsion free (because~$F' \subset F$) and has a nonscalar endomorphism.
		Hence, the induction hypothesis applies to~$F'$ and yields an endomorphism $\psi' \colon F' \longrightarrow F'$
		such that $	\Image{\psi'}$ is simple. Then we take $\phi = \psi' \circ \psi$.
	\end{proof}

	Now we are ready to construct a harmonic filtration.
	
	\begin{proposition}\label{good filtration}
		Suppose that~$F$ is a balanced vector bundle. Then it has a harmonic filtration $F = F_1 \supsetneq F_2 \supsetneq \dots \supsetneq F_{n + 1} = 0$ such that all $F_i$ are balanced.
		If~$F$ is not simple, the length $n$ of the filtration is at least~$2$.
	\end{proposition}
	
	\begin{proof}
		We argue by induction on $\Rk{F}$. If $F$ is simple there is nothing to prove.

		Suppose $F$ is not simple.
		Lemma~\ref{simple image rk leq 1/2} proves the existence of an endomorphism $\phi \colon F \longrightarrow F$
		such that $\Rk{\Image{\phi}} \leq \frac{1}{2}\Rk{F}$ and $\Image{\phi}$ is a simple torsion free sheaf.
		We set
		\begin{equation*}
		F_1 \coloneqq F,
		\qquad
		F_2 \coloneqq \Ker{\phi}
		\end{equation*}
		and check that the properties~\ref{filt:gi} and~\ref{filt:cgi} of Definition~\ref{def:good-filtration} for~$i = 1$ are satisfied.
		Indeed, \ref{filt:gi} holds by construction of~$\phi$.
		Moreover, if~$D_1 \coloneqq -c_1(F_1/F_2) = c_1(\Image\phi)$ then Lemma~\ref{l0}\ref{it:h0} applied to $G = F_1/F_2$
		(whose assumptions are satified because~$F_1/F_2 \cong \Image{\phi} \subset F$
		and~$F^\vee$ is globally generated away from a finite set)
		implies that~$D_1 > 0$ and~$h^0(D_1) \ge 2$; this proves~\ref{filt:cgi}.

		To conclude, we note that by Lemma~\ref{Mukai vector of a kernel} applied to $G = F_2$
		(whose assumptions are satisfied because~$F_1/F_2 = \Image{\phi} \subset F$ and~$h^0(F) = 0$)
		the bundle~$F_2$ is balanced, so induction applies.
	\end{proof} 

	Now we apply the results of this section to Lazarsfeld bundles (see Definition~\ref{definition of L bundle}).
	Recall from the Introduction that for a line bundle $A$ on a smooth curve $C$ of genus $g$
	the Brill--Noether number is defined as $\rho(A) = g - h^0(A)h^1(A)$.
	Some of the basic properties of Lazarsfeld bundles are summarized below.

	\begin{proposition}[Section 1, \cite{lazarsfeld1986brill}]
	\label{plazarsfeld}
		Let $A$ be a globally generated line bundle on a smooth curve $C \in |H|$
		and let $F = F_{A,C}$ be the corresponding Lazarsfeld bundle. Then
		\begin{enumerate}[label={\textup{(\arabic*)}}]
			\item
			\label{it:f-mukai}
			the Mukai vector of $F$ has the form $v(F) = \langle   h^0(A), -H, h^1(A) \rangle$;
			\item
			\label{it:f-h}
			$h^0(F) = h^1(F) = 0$ and~$h^2(F) = \chi(F) = h^0(A) + h^1(A)$;
			\item
			\label{it:f-chi}
			$\chi( F \otimes F^{\vee}) = 2 - 2\rho(A)$; in particular, if $\rho(A) < 0$ then $F$ is not simple;
			\item
			\label{it:f-gg}
			if $A^{\vee} \otimes \omega_C$ is globally generated then $F^{\vee}$ is globally generated.
		\end{enumerate}
		In particular, if both~$A$ and~$A^\vee \otimes \omega_C$ are globally generated, $F = F_{A,C}$ is a balanced vector bundle.
	\end{proposition}
	
	\begin{corollary}\label{good filtration for L bundles exists}
		Let $F = F_{A,C}$ be a Lazarsfeld bundle, where $A$ and $A^{\vee} \otimes \omega_C$ are globally generated.
		Then~$F$ has a harmonic filtration $F = F_1 \supsetneq F_2 \supsetneq \dots \supsetneq F_{n+1} = 0$ such that all $F_i$ are balanced.
		If~$F$ is not simple, the length $n$ of the filtration is at least~$2$.
	\end{corollary}
	
	\begin{proof}
		Since $A$ and $A^{\vee} \otimes \omega_C$ are globally generated,
		it follows from Proposition~\ref{plazarsfeld} that $F$ is a balanced vector bundle,
		hence the existence of a harmonic filtration follows from Proposition~\ref{good filtration}.
	\end{proof}
	
	We summarize properties of harmonic filtrations of Lazarsfeld bundles in the following lemma.
	
	\begin{lemma}\label{properties of good filtration of L bundles}
		Let $F = F_{A,C}$ be a Lazarsfeld bundle, where $A$ and $A^{\vee} \otimes \omega_C$ are globally generated.
		Let
		\begin{equation*}
		F = F_1 \supsetneq F_2 \supsetneq \dots \supsetneq F_{n+1} = 0
		\end{equation*}
		be a harmonic filtration.
		Denote $v(F_i/F_{i+1}) = \langle r_i, -D_i, s_i \rangle$ for $1 \leq i \leq n$. Then
		\begin{enumerate}[label={\textup{(\arabic*)}}]
			\item
			\label{it:h-decomposition}
			$r_1 + \dots + r_n = h^0(A)$, $D_1 + \dots + D_n = H$, and $s_1 + \dots + s_n = h^1(A)$;
			\item
			\label{it:di}
			$D_i > 0$, $h^0(D_i) \geq 2$ and $h^2(D_i) = 0$;
			\item
			\label{it:rs}
			$r_is_i \leq \frac{1}{2}D_i^2 + 1$;
			\item
			\label{it:r}
			$r_i \leq r_{i+1} + \dots + r_n$;
			\item
			\label{it:s}
			$s_n \geq 1$.
		\end{enumerate}
	\end{lemma}
	
	\begin{proof}
		Indeed, \ref{it:h-decomposition} follows from additivity of Mukai vectors and Proposition~\ref{plazarsfeld}\ref{it:f-mukai},
		while~\ref{it:di} follows from property~\ref{filt:cgi} of a harmonic filtration.
		Since all sheaves $F_i/F_{i+1}$ are simple, \ref{it:rs} follows from Lemma~\ref{formula simple sheaf Mukai vector}.
		Moreover, \ref{it:r} follows from additivity of rank and property~\ref{filt:gi} of a harmonic filtration.
		Finally, \ref{it:s} follows from Lemma~\ref{Mukai vector lemma} (because the sheaf~$F_n/F_{n+1} = F_n$ is balanced).
	\end{proof}

	\begin{lemma}\label{lsaturation}
		Let $G \subset F_{A,C}$ be a saturated subsheaf in a Lasafsfeld bundle and let $\widehat{G}$ be its saturation in $H^0(A) \otimes \mathcal{O}_X$. Then one of the following holds
		\begin{enumerate}[label = {\textup{(\arabic*)}}]
			\item either $\widehat{G} \cong G$;
			\item or $c_1(\widehat{G}) = c_1(G) + H$ and $\Coker({G\hookrightarrow \widehat{G}}) = i_* A'$ for some line bundle $A' \hookrightarrow A$ on the curve $C$.
		\end{enumerate}
	\end{lemma}
	
	\begin{proof}		
		Suppose that $\widehat{G} \ncong G$. Then there is a commutative diagram 
		\begin{center}
			\begin{equation}\label{diagram1}
				\begin{tikzcd}
					& 0 \arrow{r} & F_{A,C} \arrow{r} & H^0(A) \otimes \mathcal{O}_X \arrow{r} & i_*A \arrow{r} & 0 & \\
					& 0 \arrow{r} & G \arrow{r} \arrow[hook]{u} & \widehat{G} \arrow{r} \arrow[hook]{u} & N \arrow{r} \arrow{u} & 0 &
				\end{tikzcd}
			\end{equation}
		\end{center}
		with $N \ne 0$. We have $\Ker{(N \longrightarrow i_*A)} \subset \Coker{(G \longrightarrow F_{A,C})}$. Since $\Coker{(G \longrightarrow F_{A,C})}$ is a torsion free sheaf, because $G$ is saturated in $F = F_{A.C}$, and $\Ker{(N \longrightarrow i_*A)}$ is a torsion sheaf, as $\Rk{\widehat{G}} = \Rk{G}$, hence $\Ker{(N \longrightarrow i_*A)} = 0$, i.e. $N \hookrightarrow i_*A$. It follows, since $N \neq 0$, that $N \cong i_*A'$, where $A' \subset A$ is a line bundle on the curve $C$. We conclude that $c_1(\widehat{G}) = c_1(G) + H$.
	\end{proof}
	
	\begin{corollary}\label{saturation-corollary-for-Petri}
	 Let $G \subset F_{A,C}$ be a saturated subsheaf in a Lasafsfeld bundle such that $$c_1(G) > c_1(F_{A,C}) = -H.$$ Then $G$ is saturated in $H^0(A) \otimes \mathcal{O}_X$.
	\end{corollary}
	
	\begin{proof}
	Indeed, if $G \subset H^0(A) \otimes \mathcal{O}_X$ is not saturated then by Lemma~\ref{lsaturation} the first Chern class of its saturation $\widehat{G} \subset H^0(A) \otimes \mathcal{O}_X$ is $c_1(\widehat{G}) = c_1(G) + H > 0$ which contradicts Lemma~\ref{lem:cc}.
	\end{proof}
	
	\section{Proof of Theorem~\ref{main theorem}}
	\label{sec proof}

	In this section we combine the results from Sections~\ref{sec Divisors on Brill--Noether General K3 Surface} and~\ref{sec Lazarsfeld Bundles}
	to prove Theorem~\ref{main theorem}.

	Throughout this section, we work under the following assumptions.

	\begin{setup}
	\label{setup}
	Let~$(X,H)$ be a quasi-polarized K3 surface of genus~$g$.
	Fix a smooth curve~$C \in |H|$ and a line bundle~$A$ on~$C$
	such that both~$A$ and~$A^\vee \otimes \omega_C$ are globally generated and

	\begin{equation}
	\label{eq:rho-g}
	g \le h^0(A)h^1(A).
	\end{equation}
	Let~$F:= F_{A,C}$ be the corresponding Lazarsfeld bundle and
	let~$F = F_1 \supsetneq F_2 \supsetneq \dots \supsetneq F_n \supsetneq F_{n+1} = 0$
	be a harmonic filtration on~$F$ with~$n \ge 2$ (it exists by Corollary~\ref{good filtration for L bundles exists}).
	We consider the Mukai vectors
	\begin{equation}
	\label{eq:mv-fi}
	v(F_i/F_{i+1}) = \langle r_i, -D_i, s_i \rangle,
	\qquad
	1 \le i \le n,
	\end{equation}
	of the quotients of the harmonic filtration.
	By Lemma~\ref{properties of good filtration of L bundles} we have a decomposition $H = D_1 + \dots + D_n$,
	where for all~$i$ we have $D_i > 0$ and $h^0(D_i) \geq 2$.
	Applying Lemma~\ref{sorting fixed components} we obtain a decomposition
	\begin{equation*}
	H = \bar{D}_1 + \dots + \bar{D}_n,
	\qquad
	\text{where $\bar{D}_i > 0$ and $\bar{D}_i^2 \geq 0$}.
	\end{equation*}
	We also note that~$\bar{D}_i^2 \ge D_i^2$, hence Lemma~\ref{properties of good filtration of L bundles}\ref{it:rs} implies
	\begin{equation}
	\label{eq:ri-si-general}
			r_{i}s_{i} \leq \frac{1}{2} \bar{D}^2_{i}+ 1.
	\end{equation}
	\end{setup}
	
	To prove Theorem~\ref{main theorem} we consider separately the cases where the length of the harmonic filtration of~$F$
	is~$n = 2$, $n = 3$, $n = 4$ and $n \geq 5$;
	in each of these cases we show that the K3 surface $(X,H)$ is not Brill--Noether general.

	To begin with, let us consider the case where~$n = 2$.
	The second part of the following proposition will be used in Section~\ref{sec Petri}.
	
	\begin{proposition}\label{simple kernel and image}
		In Setup~\ref{setup} assume~$n = 2$.
		If~$H = D_1 + D_2$ is the decomposition with the divisors~$D_i$ defined by~\eqref{eq:mv-fi} then

		\begin{equation*}
		h^0(D_1)h^0(D_2) \ge h^0(A)h^1(A).
		\end{equation*}
		In particular, if~$h^0(A)h^1(A) > g$ then the K3 surface $(X,H)$ is not Brill--Noether general.

		Moreover, if~$h^0(D_1)h^0(D_2) = h^0(A)h^1(A) = g$, then
		\begin{equation}
		\label{eq:rds}
		r_1s_1 = \frac{1}{2}D_1^2 + 1, \quad
		r_2s_2 = \frac{1}{2}D_2^2 + 1, \quad
		h^{>0}(D_1) = h^{>0}(D_2) = 0, \quad
		(r_1s_2 - 1)(r_2s_1 - 1) = 0,
		\quad
		r_1 = 1.
		\end{equation}
	\end{proposition}
	
	\begin{proof}
		Recall notation~\eqref{eq:mv-fi}.

		By Lemma~\ref{properties of good filtration of L bundles}
		we have
		\begin{itemize}
		\item 
		$h^0(D_i) \ge 2$ and $h^2(D_i) = 0$;
		\item
		$r_2 \geq r_1$ and $r_1 + r_2 = h^0(A)$; hence, $r_2 \geq \frac{1}{2}h^0(A) \geq r_1$;
		\item
		$s_1 + s_2 = h^1(A)$ and~$s_2 \ge 1$;
		\item
		$r_1s_1 \leq \frac{1}{2}D_1^2 + 1 = \chi(D_1) - 1$ and~$r_2s_2 \leq \frac{1}{2}D_2^2 + 1 = \chi(D_2) - 1$.
		\end{itemize}
		Consider two cases: $s_1 \geq 1$ and $s_1 \leq 0$.
		
		First, suppose that $s_1 \leq 0$. In this case $s_2 \geq h^1(A)$. Then, as $h^2(D_1) = 0$, we have
		\begin{equation*}
		\frac{1}{2}h^0(A)h^1(A) \leq r_2s_2 < \chi(D_2) \leq h^0(D_2).
		\end{equation*}
		On the other hand, since $h^0(D_1) \ge 2$ taking~\eqref{eq:rho-g} into account, we obtain
		\begin{equation*}
		g \le h^0(A)h^1(A) = 2 \cdot \frac{1}{2}h^0(A)h^1(A) < h^0(D_1)h^0(D_2).
		\end{equation*}
		Thus, the K3 surface $(X,H)$ is not Brill--Noether general. Note that in this case the obtained inequality $$h^0(D_1)h^0(D_2) > h^0(A)h^1(A)$$ is always strict.

		Now suppose that $s_1 \geq 1$. In this case one has the inequality
		\begin{equation*}
		(r_1 + r_2)(s_1 + s_2) \leq (r_1s_1+1)(r_2s_2+ 1).
		\end{equation*}
		Indeed, expanding the parentheses and taking all the summands to the right hand side one gets the inequality
		that holds under our assumptions: $0 \leq r_1r_2s_1s_2 -r_1s_2 -r_2s_1 + 1 = (r_1s_2 - 1)(r_2s_1 - 1)$.
		Finally, as $h^2(D_1) = h^2(D_2) = 0$, we have
		\begin{equation*}
		h^0(A)h^1(A) = (r_1 + r_2)(s_1 + s_2) \leq (r_1s_1+1)(r_2s_2+ 1) \leq \chi(D_1)\chi(D_2)\leq h^0(D_1)h^0(D_2),
		\end{equation*}
		and, therefore, the K3 surface $(X,H)$ is not Brill--Norther general, again by~\eqref{eq:rho-g}.
		
		To prove the second part, note that if the equality~$h^0(D_1)h^0(D_2) = h^0(A)h^1(A)$ holds, then
		$$(r_1s_2 - 1)(r_2s_1 - 1) = 0, \qquad r_1s_1 = \frac{1}{2}D_1^2 + 1, \qquad r_2s_2 = \frac{1}{2}D_2^2 + 1, \qquad \text{and} \qquad h^1(D_1) = h^1(D_2) = 0.$$ The first equality can hold only when either $r_1 = s_2 = 1$ or $r_2 = s_1 = 1$.
		But properties of harmonic filtration imply that $r_1 \le r_2$; hence, in either case we get $r_1 = 1$.

	\end{proof}
	
	In the following proposition we consider the most complicated case $n = 3$.
	In the proof, we will use the following elementary observation:

	\begin{lemma}
	\label{lem:am-gm}
	Let~$u,v,c \in \mathbb{Z}$.
	If~$u + v \le 2c$ then~$uv \le c^2$.
	If, moreover, $0 \le v < c$ then~$uv \le c^2 - 1$.
	\end{lemma}

	\begin{proof}
	The first is the AM--GM inequality.
	To prove the second we first note that increasing~$u$ to~$2c - v$ we only increase~$uv$, so we may assume that~$u + v = 2c$.
	Then we can write~$u = c + 1 + t$, $v = c - 1 - t$ for~$t \ge 0$, and then
	\begin{equation*}
	uv = (c + 1 + t)(c - 1 - t) = c^2 - (t + 1)^2.
	\end{equation*}
	Since~$t + 1 \ge 1$, we conclude that~$uv \le c^2 - 1$, as required.
	\end{proof}
	
	As before, to prove Theorem~\ref{main theorem} we only need the non-strict inequality $h^0(H_1)h^0(H_2) \ge h^0(A)h^1(A)$, but the stronger result proved here will be used in Section~\ref{sec Petri}.
	
	\begin{proposition}\label{theorem case n=3}
		In Setup~\ref{setup} assume~$n = 3$.
		Then either one of the conditions of Proposition~\ref{conditions on a polarizing divisor} \ref{case3}
		or of Lemma~\ref{no negative intersections}\ref{it:condition-2} is satisfied
		or there is a decomposition $H = H_1 + H_2$ such that $$h^0(H_1)h^0(H_2) > h^0(A)h^1(A).$$
		In particular, the K3 surface $(X,H)$ is not Brill--Noether general.
	\end{proposition}
	 
	\begin{proof}
		Consider the divisor classes~$\bar{D}_i$ defined in Setup~\ref{setup}.
		We will have to take care of the two orders on the set of the divisor classes~$\bar{D}_i$.
		The first order comes from their appearance in the harmonic filtration of~$F$,
		and the other is the decreasing order of their self-intersections.
		Accordingly, we denote by~$(i_1,i_2,i_3)$ a permutation of~$(1,2,3)$ such that the integers
		\begin{equation*}
		\varepsilon_1 \coloneqq \frac12\bar{D}_{i_1}^2,
		\qquad
		\varepsilon_2 \coloneqq \frac12\bar{D}_{i_2}^2,
		\qquad
		\varepsilon_3 \coloneqq \frac12\bar{D}_{i_3}^2,
		\end{equation*}
		satisfy
		\begin{equation*}
		\varepsilon_1 \ge \varepsilon_2 \ge \varepsilon_3.
		\end{equation*}
		Clearly, if either of the conditions of Proposition~\ref{conditions on a polarizing divisor}\ref{case3} is satisfied,
		$X$ is not Brill--Noether general.
		Therefore, we can assume that one of the conditions of Remark~\ref{exceptional cases}\ref{case:bn-3} holds, i.e.,
		\begin{equation*}
		\varepsilon_3 = 0,
		\qquad
		\varepsilon_2 \in \{0,1\},
		\qquad\text{and}\qquad
		\text{if~$\varepsilon_2 = 1$ then~$\varepsilon_1 \in \{1,2,3\}$.}
		\end{equation*}
		Similarly, if the condition of Lemma~\ref{no negative intersections}\ref{it:condition-2} is satisfied, the K3 surface~$(X, H)$ is not Brill--Noether general; hence, we can assume that
		\begin{equation}
		\label{eq:di-dj}
		\bar{D}_{i_2} \cdot \bar{D}_{i_3} \geq
		3 + \frac{1}{2}\left(\bar{D}_{i_1}^2 + \bar{D}_{i_2}^2 + \bar{D}_{i_3}^2 \right) =
		3 + \varepsilon_1 + \varepsilon_2.
		\end{equation}

		To prove the proposition we will deduce a lower bound for the product of Euler characteristics
		\begin{align}
		\label{of the form 1}
			\chi(\bar{D}_{i_1})\chi(\bar{D}_{i_2} + \bar{D}_{i_3}) & \geq f_1(\varepsilon_1, \varepsilon_2),\\
		\intertext{and an upper bound for the product~$h^0(A)h^1(A)$}
		\label{of the form 2}
			f_2(\varepsilon_1, \varepsilon_2) &\geq h^0(A)h^1(A),
		\end{align}
		where $f_1(\varepsilon_1, \varepsilon_2)$ and $f_2(\varepsilon_1, \varepsilon_2)$
		are polynomials in $\varepsilon_1$ and $\varepsilon_2$ such that
		\begin{equation}
		\label{eq:c1-c2-inequality}
		f_1(\varepsilon_1, \varepsilon_2) \geq f_2(\varepsilon_1, \varepsilon_2)
		\qquad
		\text{for $\varepsilon_2 \in \{0,1\}$}.
		\end{equation}
		Since~$h^2(\bar{D}_{i_1}) = h^2(\bar{D}_{i_2} + \bar{D}_{i_3}) = 0$, we will obtain a chain of inequalities
		\begin{equation}
		\label{eq:chain}
		h^0(A)h^1(A) \leq
		f_2(\varepsilon_1, \varepsilon_2) \leq
		f_1(\varepsilon_1, \varepsilon_2) \leq
		\chi(\bar{D}_{i_1})\chi(\bar{D}_{i_2} + \bar{D}_{i_3}) \leq
		h^0(\bar{D}_{i_1})h^0(\bar{D}_{i_2} + \bar{D}_{i_3}),
		\end{equation}
		where the first is~\eqref{eq:rho-g}, the second is~\eqref{of the form 2}, the third is~\eqref{eq:c1-c2-inequality},
		the fourth is~\eqref{of the form 1}, and the last follows from the vanishing of~$h^2$.
		Finally, we will check that the obtained inequality $h^0(A)h^1(A) \le h^0(\bar{D}_{i_1})h^0(\bar{D}_{i_2} + \bar{D}_{i_3})$ is always strict.
		Therefore, the classes~$H_1 \coloneqq \bar{D}_{i_1}$ and $H_2 \coloneqq \bar{D}_{i_2} + \bar{D}_{i_3}$
		will have the required property	and, in particular, it will follow that the K3 surface $(X,H)$ is not Brill--Noether general.

		\medskip

		{\bf Proof of the inequality~\eqref{of the form 1}.}
		On the one hand, we have~$\chi(\bar{D}_{i_1}) = \frac{1}{2}\bar{D}_{i_1}^2 + 2 = \varepsilon_1 + 2$.
		On the other hand, using~\eqref{eq:di-dj}, we deduce
		\begin{equation*}
		\chi(\bar{D}_{i_2} + \bar{D}_{i_3}) =
		\frac{1}{2} \left( \bar{D}_{i_2} + \bar{D}_{i_3} \right)^2 + 2 \geq
		\varepsilon_2 + \left( 3 + \varepsilon_1 + \varepsilon_2 \right) + 2 =
		\varepsilon_1 + 2\varepsilon_2 + 5.
		\end{equation*}
		Multiplying these equality and inequality, we deduce
		\begin{equation}\label{actual ineq of the form 1}
			\chi(\bar{D}_{i_1})\chi(\bar{D}_{i_2} + \bar{D}_{i_3}) \geq
			\left(\varepsilon_1 + 2\right)\left(\varepsilon_1 + 2\varepsilon_2 + 5\right).
		\end{equation}
		So, we take $f_1(\varepsilon_1,\varepsilon_2) = \left(\varepsilon_1 + 2\right)\left(\varepsilon_1 + 2\varepsilon_2 + 5\right) = \varepsilon_1^2 + \left(2\varepsilon_2+7\right)\varepsilon_1 + 4\varepsilon_2 + 10$.

		\medskip

		{\bf Proof of the inequalities~\eqref{of the form 2} and~\eqref{eq:c1-c2-inequality}.}
		The inequalities~\eqref{eq:ri-si-general} can be rewritten as
		\begin{equation}
		\label{eq:ri-si}
			r_{i_1}s_{i_1} \leq \frac{1}{2} \bar{D}^2_{i_1}+ 1 = \varepsilon_1+ 1,
			\quad
			r_{i_2}s_{i_2} \leq \frac{1}{2} \bar{D}^2_{i_2}+ 1 = \varepsilon_2 + 1,
			\quad
			r_{i_3}s_{i_3} \leq  \frac{1}{2} \bar{D}^2_{i_2}+ 1 = 1,
		\end{equation}
		Since $r_i \ge 1$ it follows that~$s_{i_1} \leq \varepsilon_1 + 1$, $s_{i_2} \leq \varepsilon_2 + 1$ and $s_{i_3}\leq 1$.
		Therefore,
		\begin{equation}
		\label{eq:h1-inequality}
		h^1(A) = s_1 + s_2 + s_3 \leq \varepsilon_1 + \varepsilon_2 + 3,
		\end{equation}
		and if~\eqref{eq:h1-inequality} is an equality then $s_{i_k} = \varepsilon_{i_k} + 1 \ge 1$, i.e. $s_{i_k} \ge 1$ for all $i$.
		Note also that if $s_i \geq 1$ then
		\begin{equation}\label{12345}
			r_i+s_i \leq r_is_i+1.
		\end{equation}
		We distinguish three cases, depending on the signs of~$s_i$:
		\begin{enumerate}[label={\textup{(\alph*)}}]
		\item
		\label{case:all-positive}
		the case where all $s_i$ are positive,
		\item
		\label{case:one-non-positive}
		the case where only one of the $s_i$ is non-positive, and
		\item
		\label{case:two-non-positive}
		the case where two of the $s_i$ are non-positive.
		\end{enumerate}
		Note that the inequality $s_3 \geq 1$ always holds by Lemma~\ref{properties of good filtration of L bundles},
		so the above list of cases is exhaustive.
		
		{\bf Case~\ref{case:all-positive}.}
		Using Lemma~\ref{properties of good filtration of L bundles} and inequalities~\eqref{12345} and~\eqref{eq:ri-si}, we deduce
		\begin{equation}\label{sum case 1}
			h^0(A) + h^1(A) = r_1 + r_2 + r_3 + s_1 + s_2 + s_3 \leq r_1s_1 + r_2s_2 + r_3s_3 + 3 \leq \varepsilon_1 + \varepsilon_2 + 6.
		\end{equation}
		Therefore, Lemma~\ref{lem:am-gm} implies
		\begin{equation*}
		h^0(A)h^1(A) \leq
		\left(\frac{\varepsilon_1 + \varepsilon_2 + 6}{2}\right)^2.
		\end{equation*}
		Thus, in this case we have~\eqref{of the form 2}
		with~$f_2(\varepsilon_1,\varepsilon_2) = \left(\frac{\varepsilon_1 + \varepsilon_2 + 6}{2}\right)^2$.

		Now we check the inequality~\eqref{eq:c1-c2-inequality}:
		\begin{equation*}
		f_1(\varepsilon_1,\varepsilon_2) - f_2(\varepsilon_1,\varepsilon_2) =
		\left(\varepsilon_1 + 2\right)\left(\varepsilon_1 + 2\varepsilon_2 + 5\right) -
		\left(\frac{\varepsilon_1 + \varepsilon_2 + 6}{2}\right)^2 =
		\frac{3}{4}\varepsilon_1^2 - \frac{1}{4}\varepsilon_2^2 + \frac{3}{2}\varepsilon_1 \varepsilon_2 +
		4 \varepsilon_1 + \varepsilon_2 + 1.
		\end{equation*}
		Since~$\varepsilon_1 \geq \varepsilon_2 \geq 0$, this is positive;
		therefore, in this case the inequality~\eqref{eq:c1-c2-inequality} is strict, hence so is~\eqref{eq:chain}.

		{\bf Case~\ref{case:one-non-positive}.}
		Suppose that $s_i,s_j \geq 1$, $s_k \leq 0$, $\{ i,j,k\} = \{ 1,2,3\}$.
		Since $s_3 \geq 1$, we have~$k \neq 3$.
		Moreover, recall from Lemma~\ref{properties of good filtration of L bundles} that~$r_2 \leq r_3$ and~$r_1 \leq r_2 + r_3$.
		Therefore,
		\begin{equation*}
		r_k \leq r_i + r_j.
		\end{equation*}
		Indeed, if $k = 2$ then $r_2 \leq r_3 < r_1 + r_3$, and if $k=1$ then $r_1 \leq r_2 + r_3$.
		Since~$s_i, s_j \geq 1$ by assumption, one has the following chain of inequalities
		\begin{equation*}
		r_k \leq
		r_i + r_j \leq
		r_is_i + r_js_j \leq
		\frac{1}{2}\bar{D}_i^2 + 1 + \frac{1}{2}\bar{D}_j^2 + 1 \leq
		\frac{1}{2}\bar{D}_{i_1}^2 + 1 + \frac{1}{2}\bar{D}_{i_2}^2 + 1 =
		\varepsilon_1 + \varepsilon_2 +  2.
		\end{equation*}
		Taking into account that~$s_k \leq 0$ and using Lemma~\ref{properties of good filtration of L bundles} and~\eqref{12345}, we deduce
		\begin{equation}\label{sum case 2}
			h^0(A)+h^1(A) = r_1 + r_2 + r_3 + s_1 + s_2 + s_3 \leq r_is_i + r_js_j + 2 + r_k \leq \varepsilon_1 + \varepsilon_2 + 4 + r_k \leq 2\varepsilon_1 + 2\varepsilon_2 + 6.
		\end{equation}
		Note also that the inequality~\eqref{eq:h1-inequality} is stirct, because~$s_k \le 0$.
		Therefore, Lemma~\ref{lem:am-gm} implies
		\begin{equation}
		\label{eq:inequality-case-b}
		h^0(A)h^1(A) \leq
		(\varepsilon_1 + \varepsilon_2 + 2)(\varepsilon_1 + \varepsilon_2 + 4).
		\end{equation}
		Thus, in this case we have~\eqref{of the form 2}
		with~$f_2(\varepsilon_1,\varepsilon_2) = (\varepsilon_1 + \varepsilon_2 + 2)(\varepsilon_1 + \varepsilon_2 + 4)$.

		Now we check the inequality~\eqref{eq:c1-c2-inequality}:
		\begin{equation*}
		f_1(\varepsilon_1,\varepsilon_2) - f_2(\varepsilon_1,\varepsilon_2) =
		(\varepsilon_1 + 2)( \varepsilon_1 + 2\varepsilon_2 + 5) - (\varepsilon_1 + \varepsilon_2 + 2)(\varepsilon_1 + \varepsilon_2 + 4) =
		2 + \varepsilon_1 - 2 \varepsilon_2 -\varepsilon_2^2.
		\end{equation*}
		Since~$\varepsilon_1 \geq \varepsilon_2$ and $\varepsilon_2 \in \{0,1\}$ this is nonnegative;
		therefore~\eqref{eq:c1-c2-inequality} holds (and strict, unless $\varepsilon_1 = \varepsilon_2 = 1$).

		{\bf Case~\ref{case:two-non-positive}.}
		In this case we have $s_3 \geq 1$, $s_1, s_2 \leq 0$.
		Moreover, using Lemma~\ref{properties of good filtration of L bundles}, we obtain
		\begin{equation*}
		h^0(A) = r_1 + r_2 + r_3 \leq 4r_3
		\qquad\text{and}\qquad
		h^1(A) = s_1 + s_2 + s_3 \le s_3.
		\end{equation*}
		Taking~\eqref{eq:ri-si} into account, we deduce
		\begin{equation*}
		h^0(A)h^1(A) \le 4r_3s_3 \le 4(\varepsilon_1 + 1).
		\end{equation*}
		Thus, in this case we have~\eqref{of the form 2}
		with~$f_2(\varepsilon_1,\varepsilon_2) = 4(\varepsilon_1 + 1)$.

		Now we check the inequality~\eqref{eq:c1-c2-inequality}:
		$\varepsilon_1 + 2 > \varepsilon_1 + 1$ and~$\varepsilon_1 + 2\varepsilon_2 + 5 > 4$ implies
		\begin{equation*}
		f_1(\varepsilon_1,\varepsilon_2) - f_2(\varepsilon_1,\varepsilon_2) =
		(\varepsilon_1 + 2)(\varepsilon_1 + 2\varepsilon_2 + 5) - 4(\varepsilon_1 + 1) > 0,
		\end{equation*}
		therefore,~\eqref{eq:c1-c2-inequality} is strict, hence so is~\eqref{eq:chain}.

		{\bf The inequality~\eqref{eq:chain} is strict.}
		As we noticed above, in cases~\ref{case:all-positive} and~\ref{case:two-non-positive}
		the inequality~\eqref{eq:c1-c2-inequality} is strict, hence so is~\eqref{eq:chain}.
		The same holds in case~\ref{case:one-non-positive} unless $\varepsilon_1 = \varepsilon_2 = 1$.
		So, it remains to note that if $\varepsilon_1 = \varepsilon_2 = 1$ in case~\ref{case:one-non-positive}
		the inequality~\eqref{sum case 2} gives $h^0(A) + h^1(A) \leq 10$.
		In the proof we also used the inequality
		\begin{equation*}
		r_1 + r_2 + r_3 + s_1 + s_2 + s_3 \leq r_is_i + r_js_j + 2 + r_k.
		\end{equation*}
		Therefore, if~\eqref{eq:chain} is an equality, we also must have $r_i + s_i = r_is_i + 1$ and $r_j + s_j = r_js_j + 1$ hence the equalities $s_i = s_j = 1$ must hold.
		However, if this is the case then, as in case~\ref{case:one-non-positive} only two of $s_1,s_2,s_3$ are positive,
		we have $h^1(A) = s = s_1 + s_2 + s_3 \le s_i + s_j = 2$.
		Hence, since $h^0(A) + h^1(A) = 10$, we have
		\begin{equation*}
		h^0(A)h^1(A) \le 16
		\qquad\text{and}\qquad
		(\varepsilon_1 + \varepsilon_2 + 2)(\varepsilon_1 + \varepsilon_2 + 4) = 4 \cdot 6 = 24,
		\end{equation*}
		hence inequality~\eqref{of the form 2} (which in this case takes form~\eqref{eq:inequality-case-b}) is strict,
		and therefore~\eqref{eq:chain} is strict as well.

		\medskip

		Thus, in each of the cases~\ref{case:all-positive}, \ref{case:one-non-positive} and~\ref{case:two-non-positive}
		we proved the required chain of inequalities~\eqref{eq:chain} and checked that the resulting inequality is strict.
		In particular, in all cases the K3 surface~$(X,H)$ is not Brill--Noether general.
	\end{proof}
	
	Finally, we consider the case $n = 4$.
	
	\begin{proposition}\label{theorem case n=4}
		In Setup~\ref{setup} assume~$n = 4$.
		Then either one of the conditions of Proposition~\ref{conditions on a polarizing divisor}\ref{case2} or of Lemmas~\ref{no negative intersections}\ref{it:condition-2} or of Lemma~\ref{lemma for n = 4} is satisfied
		or there is a decomposition $H = H_1 + H_2$ such that
		\begin{equation*}
		h^0(H_1)h^0(H_2) > h^0(A)h^1(A).
		\end{equation*}
		In particular, the K3 surface $(X,H)$ is not Brill--Noether general.
	\end{proposition}
	
	\begin{proof}
		Consider the divisor classes~$\bar{D}_i$ defined in Setup~\ref{setup}.
		If the condition of Proposition~\ref{conditions on a polarizing divisor}\ref{case2} is satisfied then there is nothing to prove. Therefore, we can assume that the condition of Remark~\ref{exceptional cases}\ref{case:bn-4} holds, i.e.,
		\begin{equation*}
		\bar{D}_1^2 = \bar{D}_2^2 = \bar{D}_3^2 = \bar{D}_4^2 = 0.
		\end{equation*}
		Therefore, the inequalities~\eqref{eq:ri-si-general} take the form
		\begin{equation}
		\label{eq:ri-si-n4}
		r_is_i \leq \frac{1}{2}\bar{D}_i^2 + 1 = 1.
		\end{equation}
		Note that~$s_4 \ge 1$ by Lemma~\ref{properties of good filtration of L bundles} and~$r_4 \ge 1$,
		hence~$r_4 = s_4 = 1$.

		Since $r_i \geq 1$, we have~$s_i \leq 1$, and therefore Lemma~\ref{properties of good filtration of L bundles} gives
		\begin{equation*}
		h^1(A) = s_1 + s_2 + s_3 + s_4 \leq 4.
		\end{equation*}
		First, assume~$h^1(A) = 4$.
		Then~$s_1 = s_2 = s_3 = s_4 = 1$, hence~\eqref{eq:ri-si-n4} gives~$r_1 = r_2 = r_3 = r_4 = 1$, and therefore,
		again by Lemma~\ref{properties of good filtration of L bundles}	we obtain~$h^0(A) = r_1 + r_2 + r_3 + r_4 = 4$, so that
		\begin{equation}
		\label{eq:h0h1-first}
		h^0(A)h^1(A) = 4 \cdot 4 = 16.
		\end{equation}
		Now, assume~$h^1(A) \le 3$.
		Then, Lemma~\ref{properties of good filtration of L bundles} implies that~$h^0(A) = r_1 + r_2 + r_3 + r_4 \le 8r_4 = 8$, so that
		\begin{equation}
		\label{eq:h0h1-second}
		h^0(A)h^1(A) \le 8 \cdot 3 = 24.
		\end{equation}

		It remains to show that we can write~$H$ as a sum of two effective divisors
		with the product of Euler characteristics greater than~$24$.
		Now, we will not use the properties of harmonic filtrations anymore, so we allow ourselves to renumber the divisors classes~$\bar{D}_i$.
		First, without loosing generality we assume that
		\begin{equation*}
		\bar{D}_2 \cdot \bar{D}_3 \leq \bar{D}_2 \cdot \bar{D}_4 \leq \bar{D}_3 \cdot \bar{D}_4.
		\end{equation*}
		If~$\bar{D}_3 \cdot \bar{D}_4 \le 1$, Lemma~\ref{lemma for n = 4} proves that the K3 surface~$(X,H)$ is not Brill--Noether general.
		So, we can assume that
		\begin{equation*}
		\bar{D}_3 \cdot \bar{D}_4 \geq 2.
		\end{equation*}
		It follows immediately that~$(\bar{D}_3 + \bar{D}_4)^2 \geq 4$.
		
		Similarly, if~$\bar{D}_1 \cdot \bar{D}_2 \leq 2 + \frac{1}{2}\left(\bar{D}_1^2 + \bar{D}_2^2 + \left(\bar{D}_3 + \bar{D}_4\right)^2\right)$,
		Lemma~\ref{no negative intersections} proves that the K3 surface~$(X,H)$ is not Brill--Noether general.
		So, from now on assume that
		\begin{equation*}
		\bar{D}_1 \cdot \bar{D}_2 \geq 3 + \frac{1}{2}\left(\bar{D}_1^2 + \bar{D}_2^2 + \left(\bar{D}_3 + \bar{D}_4\right)^2\right) \geq 5.
		\end{equation*}
		Now, we have
		\begin{equation*}
		\chi(\bar{D}_1 + \bar{D}_2) = \frac{1}{2}(\bar{D}_1 + \bar{D}_2)^2 + 2 \geq 5 + 2 = 7,
		\qquad
		\chi(\bar{D}_3 + \bar{D}_4) = \frac{1}{2}(\bar{D}_3 + \bar{D}_4)^2 + 2 \geq 2 + 2 = 4.
		\end{equation*}
		Hence,
		\begin{equation*}
		\chi(\bar{D}_1 + \bar{D}_2)\chi(\bar{D}_3 + \bar{D}_4) \geq 7 \cdot 4 = 28.
		\end{equation*}
		Combining this inequality with~\eqref{eq:h0h1-first} and~\eqref{eq:h0h1-second},
		we conclude that in any case
		\begin{equation*}
		h^0(A)h^1(A) < \chi(\bar{D}_1 + \bar{D}_2)\chi(\bar{D}_3 + \bar{D}_4).
		\end{equation*}
		Therefore, the classes~$H_1 \coloneqq \bar{D}_1 + \bar{D}_2$ and $H_2 \coloneqq \bar{D}_3 + \bar{D}_4$
		have the required property	and, in particular, it follows that the K3 surface $(X,H)$ is not Brill--Noether general.
	\end{proof}

	Finally, we combine the above cases and prove the main theorem.
	
	\begin{proof}[Proof of Theorem~\ref{main theorem}]
		Assume~$C \in |H|$ is not Brill--Noether general.
		By Lemma~\ref{global generation of L bundle^*} there is a line bundle~$A$ on~$C$ such that
		both~$A$ and~$A^\vee \otimes \omega_C$ are globally generated and satisfy~$h^0(A)h^1(A) > g$.
		Let~$F = F_{A,C}$ be the corresponding Lazarsfeld bundle.
		By Corollary~\ref{good filtration for L bundles exists} it has a harmonic filtration
		which has length~$n \ge 2$ by Proposition~\ref{plazarsfeld}
		and satisfies the properties of Lemma~\ref{properties of good filtration of L bundles}.

		Thus, we are in Setup~\ref{setup}.
		If~$n = 2$, we apply Proposition~\ref{simple kernel and image},
		if~$n = 3$ we apply Proposition~\ref{theorem case n=3},
		and if~$n = 4$ we apply Proposition~\ref{theorem case n=4}.
		Finally, if~$n \geq 5$ we apply Proposition~\ref{conditions on a polarizing divisor}\ref{case1}.
		In all cases we conclude that the K3 surface~$(X,H)$ is not Brill--Noether general.
	\end{proof}

\section{On Petri Property}\label{sec Petri}

In this section we prove Theorem~\ref{Petri main}.

Let $(X,H)$ be a Brill--Noether general quasi-polarized K3 surface of genus $g(H) = \frac12 H^2 + 1$.
We fix a factorization $g(H) = r \cdot s$ where $r,s \ge 2$.
Recall the definition of special Mukai classes (Definition~\ref{defin-Mukai-class}).
The following observation is crucial for our argument.

\begin{proposition}\label{Petri simple}
		Let $(X,H)$ be a Brill--Noether general  K3 surface of genus $g = g(H) = r \cdot s$ with $r,s \ge 2$.
		Let~$C \in |H|$ be a smooth curve and let~$A$ be a line bundle on $C$ with~$h^0(A) = r$ and $h^1(A) = s$.
		If~$(X,H)$ has no special Mukai classes of type $(r,s)$ then the Lazarsfeld bundle $F_{A,C}$ is simple.
\end{proposition}

\begin{proof}
	By Theorem~\ref{main theorem} the curve $C$ is Brill--Noether general.  By Lemma~\ref{global generation of L bundle^*}
	the line bundles~$A$ and $A^{\vee} \otimes \omega_C$ are globally generated,
	hence the Lazarsfeld bundle~$F = F_{A,C}$ is well defined and its dual is globally generated by Proposition~\ref{plazarsfeld}.
	
	Suppose to the contrary that~$F$ is not simple, i.e. $h^0(F \otimes F^{\vee}) \ge 2$.
	Then $F$ has a harmonic filtration of length $n \ge 2$ by Corollary~\ref{good filtration for L bundles exists}.
	First, we consider the cases with~$n \ge 3$, and then discuss the case~$n = 2$.

	If~$n = 3$ then by Proposition~\ref{theorem case n=3}
	either the conditions of Proposition~\ref{conditions on a polarizing divisor}\ref{case3} hold
	or there is a decomposition $H = H_1 + H_2$ with $h^0(H_1)h^0(H_2) > h^0(A)h^1(A) = g$;
	in either case $(X,H)$ is not Brill--Noether general.
	The same argument applies if $n = 4$ (by Proposition~\ref{theorem case n=4})
	and $n \ge 5$ (by Proposition~\ref{conditions on a polarizing divisor}\ref{case1}).

	Therefore, we may assume $n = 2$. We will show that in this case the K3 surface $X$
	contains a special Mukai class thus arriving to a contradiction with the assumption of the proposition.

	Consider the decomposition~$H = D_1 + D_2$ constructed in Proposition~\ref{simple kernel and image}.
	We have
	\begin{equation*}
	g \ge h^0(D_1)h^0(D_2) \ge h^0(A)h^1(A) = g,
	\end{equation*}
	hence the second part of Proposition~\ref{simple kernel and image} applies
	and the Mukai vectors~$(r_1,-D_1,s_1)$ and~$(r_2,-D_2,s_2)$ of the factors~$F/F_2$ and~$F_2$
	of the harmonic filtration of~$F$ satisfy equalities~\eqref{eq:rds}.
	In particular, we have
	\begin{equation}
	\label{eq:d1h-d2h}
	D_1 \cdot H = r_1s + s_1r - 1,
	\qquad
	D_2 \cdot H = r_2s + s_2r - 1.
	\end{equation}
	Indeed, $D_i^2 = 2r_is_i - 2$ by~\eqref{eq:rds}, while~$(D_1 + D_2)^2 = H^2 = 2(r_1 + r_2)(s_1 + s_2) - 2$,
	and using these equalities it is easy to find that~$D_1 \cdot D_2 = r_1s_2 + r_2s_1 + 1$ and then~\eqref{eq:d1h-d2h} follows.
	Moreover, \eqref{eq:rds} implies

	\begin{equation*}
	h^{>0}(D_i) = 0,
	\qquad\qquad\qquad
	h^0(D_i) = \chi(D_i) = r_is_i + 1.
	\end{equation*}

	and~$(r_1s_2 - 1)(r_2s_1 - 1) = 0$, hence either $r_1 = s_2 = 1$ or $r_2 = s_1 = 1$.

	Suppose that $r_1 = s_2 = 1$.
	We will show that in this case $D_2$ is a special Mukai class of type~$(r,s)$.
	First, we have~$r_2 = r - r_1 = r - 1$ and $s_1 = s - s_2 = s - 1$, hence~$h^0(D_2) = r_2s_2 + 1 = r$.
	Moreover, \eqref{eq:d1h-d2h} implies~$D_2 \cdot H = (r - 1)s + r - 1 = (r - 1)(s + 1)$.

	Finally, since $F_2 \subset F$ is a saturated subsheaf, Lemma~\ref{l0}~\ref{it:h0}
	implies that $|-c_1(F_2)| = |D_2|$ is base-point-free.
	Thus, we conclude that~$D_2$ is a special Mukai class of type $(r,s)$.
	
	Now consider the case $r_2 = s_1 = 1$.
	Note that $F_2 \subset F$ is saturated of rank~$r_2 = 1$, hence a line bundle,
	and~$F/F_2$ is a torsion-free sheaf of rank~$r_1 = 1$ by~\eqref{eq:rds}.
	Therefore, the harmonic filtration takes the form of an exact sequence
	\begin{equation}\label{Petri exact sequence}
		0 \longrightarrow \mathcal{O}(-D_2) \longrightarrow F \longrightarrow \mathcal{O}(-D_1) \otimes I_Z \longrightarrow 0,
	\end{equation}
	where $Z$ is a zero-dimensional subscheme.
	On the other hand, the construction of harmonic filtration in Proposition~\ref{good filtration}
	shows that one of the following two properties is satisfied:
	\begin{itemize}
	\item
	$\mathcal{O}(-D_1) \otimes I_Z \hookrightarrow \mathcal{O}(-D_2)$, or
	\item
	$F \cong \mathcal{O}(-D_2) \oplus \mathcal{O}(-D_1) \otimes I_Z$
	\end{itemize}

	In the first case~$\mathcal{O}(D_2) \hookrightarrow \mathcal{O}(D_1)$.
	Thus, $r_2s_2 + 1 = h^0(D_2) \le h^0(D_1) = r_1s_1 + 1 = 2$.
	Since $r_2 = 1$ and $s_2 = s - s_1 = s - 1 \ge 1$, we conclude that $s_2 = 1$
	and we are in the previous case where $r_1 = s_2 = 1$; in particular, $D_2$ is a special Mukai class of type $(r,s)$.

	In the second case it follows that~$Z = \varnothing$ because $F$ is locally free.
	Therefore, $F \cong \mathcal{O}(-D_1) \oplus \mathcal{O}(-D_2)$
	and we conclude that $\mathcal{O}(-D_1) \subset F$ is a saturated subsheaf, so $|D_1|$ is base-point-free by Lemma~\ref{l0}~\ref{it:h0}.
	As before, $h^0(D_1) = 2 = r$,  and~\eqref{eq:d1h-d2h} implies~$D_1 \cdot H = 2 + 2 - 1 = 3 = (r - 1)(s + 1)$,
	hence~$D_1$ is a special Mukai class of type $(r,s)$.
	\end{proof}
	
	For the proof of Theorem~\ref{Petri main}, we need the following construction from~\cite{lazarsfeld1986brill}.
	
	\begin{definition}[\cite{lazarsfeld1986brill}, Definition 2.1]
		Fix a vector space $V$ of dimension $r$.
		Let $P_{(r-1)(s+1)}^{(r-1)}$ be the quasi-projective scheme parametrizing the set of all triples $\left(C,A,\lambda\right)$, where:
		\begin{enumerate}[label={\textup{(\arabic*)}}]
 			\item $C \subset X$ is a smooth curve in~$|H|$,

 			\item
 			$A$ is a line bundle on $C$ such that $\Deg{A} = (r-1)(s+1)$, $\dim{H^0(A)} = r$,
 			and both~$A$ and~$A^\vee \otimes C$ are globally generated.
			
			\item $\lambda$ is a surjective homomorphism of sheaves: $$ \lambda : V \otimes \mathcal{O}_X \longrightarrow i_{*}A $$ inducing an isomorphism $V \cong H^0(A)$, where~$i \colon C \hookrightarrow X$ is the emnedding of~$C$.
		\end{enumerate}
		Triples~$(C_1,A_1,\lambda_1)$ and~$(C_2,A_2,\lambda_2)$ are identified if~$C_1 = C_2$, $A_1 \cong A_2$,
		and~$\lambda_1$ and~$\lambda_2$ are proportional.
	\end{definition}
	
	The scheme~$P_{(r-1)(s+1)}^{(r-1)}$ is constructed in~\cite{lazarsfeld1986brill},
	and by construction it is endowed with a morphism
	\begin{equation*}
	\pi \colon P_{(r-1)(s+1)}^{(r-1)} \longrightarrow |H|,
	\qquad
	\left(C,A,\lambda\right) \mapsto C.
	\end{equation*}
	We are going to use the following result.
	\begin{proposition}[\cite{lazarsfeld1986brill}, Proposition 2.2]\label{L proposition 2.2}
		Fix any point $\left(C, A, \lambda\right) \in P_{(r-1)(s+1)}^{(r-1)}$ and let $F = \ker{\lambda}$.
		Assume that $h^0\left(F \otimes F^{\vee}\right) = 1$.
		Then
		\begin{enumerate}[label={\textup{(\arabic*)}}]
			\item The scheme $P_{(r-1)(s+1)}^{(r-1)}$ is smooth at $\left(C,A, \lambda \right)$;
			\item The map $\pi$ is smooth at $\left(C,A, \lambda \right)$ if and only if the Petri homomorphism
			$$\mu_0: H^0(A) \otimes H^0(A^{\vee} \otimes \omega_C ) \longrightarrow H^0(\omega_C )$$
			is injective.
		\end{enumerate}
	\end{proposition}
	
	With Proposition~\ref{L proposition 2.2}, the proof of Theorem~\ref{Petri main} is quite simple. 
	
	\begin{proof}[Proof of Theorem~\ref{Petri main}]
	Let~$C$ be any smooth curve in~$|H|$; then~$C$ is Brill--Noether general by Theorem~\ref{main theorem}.
	Note that~$\pi^{-1}(C)$ is not empty by the existence theorem of Brill--Noether theory and Lemma~\ref{global generation of L bundle^*}.
	Moreover, if~$(C,A,\lambda) \in \pi^{-1}(C)$, then~$\lambda$ is the evaluation homomorphism for~$i_*A$,
	hence~$F = \ker \lambda \cong F_{A,C}$ is the Lazarsfeld bundle.
	By Proposition~\ref{Petri simple} it is simple, hence Proposition~\ref{L proposition 2.2} applies
	and we conclude that the preimage in~$P_{(r-1)(s+1)}^{(r-1)}$ of the locus of smooth curves in~$|H|$ is nonempty and smooth.

	Therefore, by Generic Smoothness Theorem, the morphism $\pi$ is smooth
	over an open subset $U \subset |H|$ contained in the locus of smooth curves.
	Invoking Proposition~\ref{L proposition 2.2} again, we see that the Petri homomorphism $$\mu_0: H^0(A) \otimes H^0(A^{\vee} \otimes \omega_C ) \longrightarrow H^0(\omega_C )$$ is injective for any extremal line bundle $A$ on any smooth curve $C \in U$. Since for such line bundles the dimension of both sides of $\mu_0$ is the same and equals $g = rs$, it follows that $\mu_0$ is an isomorphism.
	\end{proof}
	
	Note that the absence of the Mukai special class is only used in Proposition~\ref{Petri simple}.


\begin{thebibliography}{BKM25}
	
\bibitem[ACGH]{MR770932}
E. Arbarello, M. Cornalba, P.~A. Griffiths,
J. Harris.
\newblock Geometry of algebraic curves. {V}ol. {I}
\newblock Springer-Verlag, New York, 1985

\bibitem[BKM24]{bayer2024}
A. Bayer, A. Kuznetsov, E. Macrì.
\newblock Mukai bundles on Fano threefolds. 
\newblock Compos. Math. 162 (2026), no. 1, 59--99. 

\bibitem[BKM25]{bayer2025mukai} 
A. Bayer, A. Kuznetsov, E. Macrì.
\newblock Mukai models of Fano varieties. 
\newblock J. Reine Angew. Math. 836 (2026), 111--162.

\bibitem[GH80]{MR563378}
P. Griffiths, J. Harris.
\newblock On the variety of special linear systems on a general algebraic
  curve.
\newblock Duke Math. J., 47(1):233--272, 1980.

\bibitem[GLT15]{greer2015picard}
F. Greer, Z. Li, Z. Tian.
\newblock Picard groups on moduli of K3 surfaces with Mukai models.
\newblock International Mathematics Research Notices,
  2015(16):7238--7257, 2015.

\bibitem[Hab24]{Hab}
R. Haburcak.
\newblock Curves on {B}rill-{N}oether special {K}3 surfaces.
\newblock Math. Nachr., 297(12):4497--4509, 2024.

\bibitem[HL10]{huybrechts2010geometry}
D. Huybrechts, M. Lehn.
\newblock The geometry of moduli spaces of sheaves.
\newblock Cambridge University Press, 2010.

\bibitem[JK04]{JK}
T. Johnsen, A.~L. Knutsen.
\newblock {$K3$} projective models in scrolls, volume 1842 of Lecture Notes in Mathematics.
\newblock Springer-Verlag, Berlin, 2004.

\bibitem[Laz86]{lazarsfeld1986brill}
R. Lazarsfeld.
\newblock Brill--Noether--Petri without degenerations.
\newblock Journal of Differential Geometry, 23(3):299--307, 1986.

\bibitem[Muk02]{Mukai}
S. Mukai.
\newblock New developments in the theory of {F}ano threefolds: vector bundle
  method and moduli problems [translation of {S}\={u}gaku {\bf 47} (1995), no.
  2, 125--144].
\newblock volume~15, pages 125--150. 2002.
\newblock Sugaku expositions.

\bibitem[SD74]{72768972-f7b8-3f72-b8d3-5e64ee0f3017}
B.~Saint-Donat.
\newblock Projective models of K3 surfaces.
\newblock American Journal of Mathematics, 96(4):602--639, 1974.

\end{thebibliography}
\end{document}